\documentclass{amsart}

\usepackage{latexsym}
\usepackage{amsmath}
\usepackage{marvosym}
\usepackage{amsthm}
\usepackage{amscd}
\usepackage{graphicx}
\usepackage{amssymb}
\usepackage{xspace}
\usepackage{color}

\newtheorem{theorem}{Theorem}[section]
\newtheorem{lemma}[theorem]{Lemma}

\newtheorem{proposition}[theorem]{Proposition}

\newtheorem{corollary}[theorem]{Corollary}

\theoremstyle{definition}
\newtheorem{definition}[theorem]{Definition}

\newtheorem{cor}[theorem]{Corollary}

\theoremstyle{remark}
\newtheorem{remark}[theorem]{Remark}

\numberwithin{equation}{section}

\newcommand\vare{\varepsilon}
\newcommand\onto{\twoheadrightarrow}

\newcommand\EE{\mathbb E}

\newcommand\NN{\mathbb N}
\newcommand\RR{\mathbb R}
\newcommand\ZZ{\mathbb Z}

\newcommand\cB{\mathcal{B}}

\newcommand\cG{\mathcal{G}}

\newcommand\cH{\mathcal{H}}
\newcommand\cHH{{\hat{\mathcal H}}}

\newcommand\cS{\mathcal{S}}

\newcommand\cU{\mathcal{U}}

\newcommand\cK{\mathcal{K}}
\newcommand\cF{\mathcal{F}}

\renewcommand\b{\beta}

\newcommand\eps{\epsilon}
\newcommand\h{{\mathfrak h}}
\renewcommand\t{{\mathfrak t}}
\newcommand\s{{\mathfrak s}}
\renewcommand\k{{\mathfrak k}}
\newcommand\m{{\mathfrak m}}

\newcommand\hh{{\hat{\h}}}
\newcommand\hk{{\hat{\k}}}

\newcommand\hm{{\hat{\m}}}

\newcommand\ux{{\mathcal{U}(X)}}

\newcommand\nonterm{{\buildrel{\circ}\over{B}}}
\newcommand\bp{{\mathbf o}}

\newcommand\Aut{\operatorname{Aut}}

\newcommand\stab{{\rm stab}}
\newcommand\norm[1]{\left\|#1\right\|}

\newcommand\abs[1]{\left|#1\right|}

\newcommand\set[1]{\left\{{#1}\right\}}

\renewcommand\paragraph[1]{{\bigskip\noindent{\bf #1}.}}

\newcommand{\cat}{{\upshape CAT(0)}\xspace} 
\def\nonterm{{\cU_{NT}}}

\begin{document}
\title{The Poisson boundary of \cat cube complex groups}

\author{Amos Nevo}
\address{Department of Mathematics, Technion}
\email{anevo@tx.technion.ac.il}
\thanks{The first author was  supported 
in part by ISF  grant}

\author{Michah Sageev}
\address{Department of Mathematics, Technion}
\email{sageevm@tx.technion.ac.il}
\thanks{The second author was supported in part by ISF Grant}

\subjclass{}

\date{8 May 2011}


\keywords{}

\begin{abstract}
We consider a finite-dimensional, locally finite \cat  cube complex $X$ admitting a co-compact properly discontinuous  countable group of automorphisms $G$. We construct a natural compact metric space $B(X)$ on which $G$ acts by homeomorphisms,  the action being minimal and strongly proximal. Furthermore, for any generating probability measure on $G$, $B(X)$ admits a unique stationary measure, and when the measure has finite logarithmic moment, it constitutes a compact metric model of the Poisson boundary. We identify a dense $G_\delta$ subset  $\nonterm(X)$ of $B(X)$ 
 on which the action of $G$ is Borel-amenable, and describe the relation of these two spaces to the Roller boundary. Our construction can be used to give a simple geometric proof of Property A for the complex. 
Our methods are based on direct geometric arguments regarding the asymptotic behavior of half-spaces and their limiting ultrafilters, which are of considerable independent interest. In particular we analyze the notions of median and interval in the complex, and use the latter in the proof that $B(X)$ is the Poisson boundary
via the strip criterion developed by V. Kaimanovich \cite{k}. 
\end{abstract}

\maketitle


\section{Introduction}

 Groups appearing as lattices in the automorphism groups of  \cat cube complexes  have been a subject of considerable interest in recent years. On the one hand, the class of such groups include a broad spectrum of groups from across geometric group theory, 
 including Coxeter groups, right-angled Artin groups, certain arithmetic lattices in real hyperbolic space, as well as  small cancellation groups. On the other hand, the combinatorial nature of \cat cube complexes allows one to develop techniques and results that are sharper than those that hold in the general context of \cat spaces. The present paper is devoted to developing aspects of boundary theory for lattices in \cat cube complexes,
exhibiting useful analogies with boundary theory for lattices in semisimple Lie groups. 

To motivate this analogy, note that given a lattice subgroup of a connected semisimple Lie group with finite center, or more generally any discrete subgroup, an important tool in understanding its properties is the study of its actions on various compact homogeneous spaces of the Lie group, called boundary spaces.  Taking $SL_n(\RR)$ as an example, one considers the action of a  subgroup on the maximal boundary, namely the space $\cF_n$ of full flags on $\RR^n$, as well as on its equivariant factor spaces, the spaces of partial flags. Generally, for any semisimple algebraic group $H$, the maximal boundary of a semismple group $H$ is defined as the compact homogeneous space $H/P$, where $P$ is a minimal parabolic subgroup, and its factors are given by $H/Q$, where $Q$ is a parabolic subgroup containing $P$. For a lattice subgroup,  the action on the maximal boundary is minimal and  strongly proximal, and every generating probability measure on the lattice has a unique stationary measure. When the measure has finite logarithmic moment, the maximal boundary with the associated stationary measure constitutes a compact metric model of the Poisson boundary. In addition the maximal boundary mean-proximal and universally amenable action of $G$. For general discrete subgroups $G$ one can establish several of these properties (or natural modifications thereof) in considerable generality.
We recall that some uses of the  properties of the boundary action of discrete subgroups include  
\begin{enumerate}
\item The Tits alternative.
\item Simplicity and exactness of the reduced $C^\ast$-algebra.
\item Property $A$, a-T-menability, Baum-Connes and Novikov conjectures.
\item Patterson-Sullivan theory.
\item Super-rigidity of measure-preserving actions of higher rank lattices. 
\item Classification of boundary factors and normal subgroup theorem  for higher rank lattices.

\end{enumerate}

The first step in the systematic development of boundary theory and its applications for \cat cube complexes and their uniform lattices is the construction of the right notion of maximal boundary, and this is our goal in the present paper. 
We note that for general \cat cube complexes there already exist two natural compactifications. First, one can consider the usual visual boundary defined for any \cat complex, and second, specifically for cube complexes, one can consider the space consisting of all the ultrafilters on the partially ordered set of half-spaces, an  important compactification introduced by Roller \cite{Ro}. In general, however, neither 
of these spaces give rise to minimal or proximal actions, and the same applies to the constructions  discussed in \cite{Ka} and \cite{Ca07}.

We will give a direct geometric construction of a boundary space, which in effect singles out a compact invariant subset $B(X)$ of the Roller boundary of a cube complex $X$. The compact metric space $B(X)$  plays a role analogous to the unique compact orbit $H/P$ of a semisimple Lie group in the Satake compactification. Namely, $B(X)$ is a limit set for the action of the lattice on the larger Roller compactification, and gives rise to a minimal and strongly proximal action. For any generating measure on the lattice, the action is uniquely stationary, and in addition it realizes the Poisson boundary when the measure has  finite logarithmic moment. $B(X)$ is also is a mean proximal space for the group. We will 
identify a dense $G_\delta$ subset denoted $\nonterm(X)$ contained in $B(X)$,  on which the action is Borel-amenable, and hence universal (or measure-wise) amenable. Furthermore, $\nonterm(X)$ has measure one with respect to any stationary measure as above.  
We will use $\nonterm(X)$ to give a simple geometric proof  that  cube complexes satisfy Yu's Property $A$ in our context (see \cite{BC+} for the general case).

We note that $B(X)$ possesses, for general \cat cube complexes,  several further important structural features analogous to those of the maximal boundary of semisimple groups.  In particular, utilizing the recent product decomposition theorem for cube complexes established in \cite{CS}, the boundary $B(X)$ can be represented as a canonical product $\prod_{i=1}^r B(X_i)$ where each $X_i$ is an essential, irreducible non-Euclidean cube complex, with the action on $B(X)$ being the direct product of the action of $G$ on $B(X_i)$. The invariant $r$ appearing here will be called the split rank of the cube complex. It is  the natural generalization of the split rank  associated with a semisimple algebraic group over a local field, in the case that its Bruhat-Tits building is a product of trees, or equivalently, in the case it has no simple factor groups of split rank at least two. In particular our construction yields $2^r$ continuous equivariant boundary factors of the Poisson boundary $B(X)$ (including $B(X)$ and the trivial factor), in analogy with the $2^r$ boundary factors of a semisimple Lie group of real-rank $r$. 

\section{Basics on \cat cube complexes}

\subsection{Hyperplanes and halfspaces}
We recall basic terminology and facts about {\it \cat} cube complexes, referring for more details to    \cite{ChN}, \cite{Gu}, \cite{N},  \cite{Ro},  \cite{S1}.

\begin{definition}
A {\it \cat cube complex} is a simply-connected
combinatorial cell complex
whose closed cells are Euclidean $n$-dimensional cubes $[0,1]^n$
of various dimensions such that:
\begin{enumerate}
\item Any two cubes either have empty intersection or intersect in a single face of each.
\item The link of each $0$-cell is a {\em flag complex},
a simplicial complex such that any $(n+1)$ adjacent vertices
belong to an $n$-simplex.
\end{enumerate}
\end{definition}
Since an $n$-cube is a product of $n$ unit intervals, each $n$-cube comes equipped with $n$ natural  projection maps to the unit interval. A {\it hypercube} is the preimage of $\{\frac12\}$ under one of these projections; each $n$-cube contains $n$ hypercubes.
A {\it hyperplane} in a \cat cube complex $X$ is a subspace
intersecting each cube in a hypercube or trivially.
 Hyperplanes are said to {\it cross} if they intersect non-trivially; otherwise they are said to be {\it disjoint}.
 The {\it carrier} of a hyperplane is the union of all cubes intersecting it. 

Here are some basic facts about hyperplanes in \cat cube complexes which we will use throughout our arguments.

\paragraph{Basic  Properties:}
\begin{enumerate}
\item each hyperplane is embedded (i.e. it intersects a given cube in a single hypercube).
\item each hyperplane separates the complex into precisely two components, called {\it half-spaces}.
\item every collection of pairwise crossing hyperplanes has a non-empty intersection.
\item each hyperplane is itself a \cat cube complex.
\end{enumerate}

We shall use the $\ell_1$ metric on $X$, namely simply the metric on the vertices which assigns to two vertices the number of hyperplanes separating them. This metric is equivalent to the path metric on the 1-skeleton of $X$. 

\subsection{Ultrafilters and the Roller boundary} Let $\cH$ denote the collection of halfspaces and $\cHH$ denote the collection of hyperplanes. The collection of halfspaces comes equipped with a natural involution $\h\mapsto \h^*$, where $\h$ and $\h^*$ are the two complementary components of a given hyperplane. We denote by $\hh$ the hyperplane associated to the halfspace $\h$.  

Recall that an {\it ultrafilter} on $\cH$ is a subset  $\alpha$ of $\cH$ satisfying the following two conditions:

\begin{enumerate} \label{Ultrafilters}
\item{\bf Choice.} For every hyperplane $\hh$, either $\h\in\alpha$ or $\h^*\in\alpha$ but not both.
\item{\bf Consistency.} If $\h\in\alpha$ and $\h\subset\h^\prime$, then $\h^\prime\in\alpha$. 
\end{enumerate}

Sometimes we will want to construct an ultrafilter and this will be done by making a consistent choice of halfspaces. This means choosing halfspaces of $\cH$ according to $(1)$ and $(2)$ above.

Given two ultrafilters $\alpha$ and $\beta$  and a hyperplane $\hh$, we say that $\alpha$ and $\beta$ are {\it separated} by $\hh$ if  $\h\in\alpha$ and $\h^*\in\beta$ or $\h\in\beta$ and $\h^*\in\alpha$.

We denote by $\cU=\ux$  the collection of all ultrafilters on the collection of halfspaces $\cH$ of $X$. There is a natural embedding of the vertex set $X^{(0)}$ of $X$ into $\cU$, namely:

$$X^{(0)} \to \cU$$
$$\hskip 70pt v\hskip 5pt \mapsto \hskip 5pt\alpha_v=\{\h\in\cH \vert v\in\h\}$$
Every vertex of $v\in X$  may be viewed as ultrafilter, namely the collection of all those hyperplanes $\h$ such that $v\in\h$. We will use $\alpha_v$ to denote the ultrafilter associated to $v$. 
As noted by \cite{Gu}, when $X$ is finite dimensional, such ultrafilters are characterized by the Descending Chain Condition (DCC), namely  every descending chain of halfspaces $\h_1\supset\h_2\supset\ldots$ terminates. Such ultrafilters are called \emph{principal} ultrafilters.

 Following Roller, we may view $\cU$ as a compactification of $X^{(0)}$ in the following simple manner.  We consider the product:

$${\mathcal P}=\prod_{\hh\in\cHH} \{\h,\h^*\}$$

Since an ultrafilter is a choice for each pair $(\h,\h^*)$, we have that $\cU\subset {\mathcal P}$. The space ${\mathcal P}$ endowed with the Tychonoff topology is a compact space. It is not difficult to see then  that $\cU$ is a closed subset of ${\mathcal P}$ and is therefore compact. Moreover, Roller shows that $X^{(0)}$ is open and dense in $\cU$. 
We thus conclude that  $\cU \setminus X^{(0)}$  constitutes a compactification of $X^{(0)}$, which we refer to as the Roller boundary. 

To develop a better understanding of what $\ux$ looks like, it is useful to recall how one metrizes the product topology.  Recall that if $Y =\prod_{n=1}^\infty X_n$ is a countable product of metric spaces of uniformly bounded diameter, we may metrize the product topology as follows. Let  $f:\NN\rightarrow \RR_+$ be any decreasing positive function such that $\lim_{n\to\infty} f(n) = 0$. Then given $\mathbf{x}=(x_n), \mathbf{y}=(y_n)\in Y$, we set 

$$d(\mathbf{x},\mathbf{y}) = \sup\set{ f(n)d(x_n,y_n)\vert n>0 }$$

In our setting, the space $Y=2^\cH$ is a countable product of two element spaces, each containing the two hafspaces associated to a given hyperplane.  We simply need to describe a function $f:\cHH\to\RR$ as above. When $X$ is proper, one can order the hyperplanes using the metric on $X$. Pick a base vertex $\bp\in X$. For any hyperplane $\hh$, we let $f(\hh)=1/d(\hh,\bp)$, where $d(\hh,\bp)$  denotes the 1-skeleton distance between $\bp$ and $\hh$. That is, 

$$d(\hh,\bp) = \abs{ \{\text{hyperplanes separating } \hh  \text{ from } \bp\} } + 1 $$

 Since there are only finitely many hyperplanes a given distance from $\bp$, we have that $f$ is a decreasing function approaching $0$, as required. Explicitly, we  have for  two distinct ultrafilters $\alpha\neq \beta$
  
$$d(\alpha,\beta)=\sup\ \set{ 1/{d(\hh,\bp)}\ \vert\ \hh\ \text{separates}\ \alpha\ \text{and}\ \beta }$$

Note also that for each halfspace $\h$ we can define the following subset of $\cU$ : 

$$U_\h=\{\alpha\in\ux\vert \h\in\alpha\}$$

We refer to such a subset as an $\cU$-halfspace. The collection of all $\cU$-halfspaces  forms a sub-basis for the Tychonoff topology on $\cU$. Thus a basic open set consists of the intersection of finitely many $\cU$-halfspaces. 

\subsection{Quotients}
One can abstract the above construction in the following way. A \emph{pocset} is a poset $\Sigma$ with an order reversing involution  $*:\Sigma\to \Sigma$. The pocset $\Sigma$ is said to have the $\emph{finite interval condition}$ if for every $A<B$, there exist finitely many $C$ such that $A<C<B$. A pair of elements of $A,B\in \Sigma$ are said to be transverse if $A$ and $A^*$ are incomparable with $B$ and $B^*$. The pocset $\Sigma$ is said to have \emph{finite width} if there is a universal bound on the size of a collection of incomparable elements. If one starts with a finite dimensional cube complex, the collection of halfspaces forms a pocset which satisfies both the finite width condition and the finite interval condition. 

Given any poset $\Sigma$ with an order-reversing involution, one can consider the collection of all ultrafilters $\cU(\Sigma)$ on $\Sigma$ as in Subsection \ref{Ultrafilters}. If $\Sigma$ satisfies the finite interval condition and has finite width, then the collection of all principal ultrafilters is the vertex set of a finite dimensional \cat cube complex $X(\Sigma)$. As noted by Roller, this construction is natural in that if one starts with a \cat cube complex $X$, considers its pocset of haflspaces $\Sigma=\cH(X)$ and then considers the cube complex whose vertices are the principal ultrafilters on $\Sigma$, then $X(\Sigma)=X$. 

Let $\cH(X)$ denote the halfspaces of $X$. Suppose that $\cK\subset\cH$ is a subset of $\cH$ closed under involution.  we can consider the ultrafilters on the halfspace system $\cK$, which we denote $\cU_\cK(X)$. 

The collection of principal ultrafilters of $\cU_\cK(X)$  is the vertex set of a \cat cube complex $X_{\cK}(X)$. There is then a natural projection map. 

$$\cU(X)\to \cU_{\cK}(X)$$
$$\alpha\mapsto \alpha\cap\cK$$ 

This projection restricts to a projection on the principal ultrafilters, so that one has a projection map $X\to X_\cK$.  

\begin{lemma}\label{Quotient}
Let $X$ be finite dimensional \cat cube complex. Then the natural projection $X\to X_\cK$ is surjective.
\end{lemma}

\begin{proof}
We need to show that given a consistent choice of half spaces for $\cK$ satisfying  DCC, it can be extended to a consistent choice of halfspaces for half spaces in $\cH$ satisfying  DCC. We do this by induction on the dimension on the dimension of $X$. 

So let $\alpha$ be a principal ultrafilter on $\cK$. We wish to extend it to an ultrafilter on $\cH$. 

When $X$ is 1-dimensional, let $\k$ be a minimal element in $\alpha$. The halfspace $\k$ is associated to some edge $e$. Let $v$ be a vertex which is the endpoint $e$ contained in $\k$. It is now clear that the ultrafilter $\alpha_v$
is the desired ultrafilter. 

We now proceed by induction. Again letting $\alpha$ be a principal ultrafilter on $\cK$, we choose a minimal halfspace $\h_0\in\alpha$. Now the hyperplanes of $\hat\cH$ are divided into those that meet 
$\hh_0$ and those that do not. We focus first on those that meet $\hh_0$. These correspond to hyperplanes in $\hh_0$, viewed as a \cat cube complex in its own right. By induction,  for  the collection of halfspaces associated to these hyperplanes, we have that there exists some vertex $v\in\hh_0$ for which 

$$\alpha_v=\{\h \vert v\in\h\ \text{ and } \hh\cap\hh_0\not=\emptyset\}$$

We now let $e$ be the edge of $X$ whose midpoint is $v$ and let $w$ be the endpoint of $e$ which is contained in $\h_0$. Then $\alpha_w$ is our desired ultrafilter. 
\end{proof}

We then have the following useful corollary.

\begin{corollary}\label{NonTrivialIntersection}
Let $X$ be a finite dimensional \cat cube complex and let $\alpha\subset\cH$ be a subset satisfying the choice and consistency conditions and satisfying DCC, then 
$$\bigcap_{\h\in\alpha}\h\not=\emptyset.$$
\end{corollary}

\begin{proof}
Let $\cK=\{\h\in\cH\vert \h\in\alpha \text{ or } \h^*\in\alpha\}$. Then $\cK$ is an involution invariant subset and we may construct the space $X_\cK$. The subset $\alpha$ is now an ultrafilter on $\cK$ and since it satisfies DCC, it corresponds to a vertex in $X_\cK$. Now by Lemma \ref{Quotient} the map $X\to X_\cK$ is surjective, so there exist vertices of $X$ mapping to $\alpha$. These vertices lie in $\bigcap_{\h\in\alpha}\h,$ as required. 
\end{proof}

A particular example of this type of projection occurs when one eliminates a single hyperplane: $\cK=\cH - \{\h,\h^*\}$.  In $X$, the collection of all cubes intersecting $\hh$ is called the \emph{carrier} of $\hh$, which we denote $C(\hh)$,  and naturally has a product structure $\hh\times [0,1]$. One then has a natural collapsing map $\hh\times [0,1]\to \hh\times \{0\}$. One can apply this to the carrier of $\hh$ to obtain a quotient $\overline{X}$  of $X$, whose hyperplanes are $\cH - \{\hh\}$. This quotient space $\overline{X}$ is the cube complex associated to the pocset $\cH - \{\h,\h^*\}$. In this instance we call the map $X\to \overline{X}$ a \emph{collapsing} map. 
We also use the term collapsing for the quotient obtained by removing finitely many halfspaces.

\subsection{Pruning}
When a group acts on a tree, it is useful to pass to a minimal invariant subtree. For a \cat cube complex, there is a similar process, described in \cite{CS}, which we now describe. A half-space is called \emph{deep} if it contains arbitrarily large balls. A hyperplane is called \emph{essential} if both of its associated halfspaces are deep. A cube complex is called \emph{essential} if all of its hyperplanes are essential.

We then have a the following result (described in more detail in \cite{CS})

\begin{theorem}
Let $X$ be a \cat cube complex with cocompact automorphism group. Then there exists a canonical essential \cat cube complex $X_{ess}$ and a $G$-equivariant map $f:X\to X_{ess}$ such that
\begin{itemize}
\item the preimage under $f$ of the collection of hyperplanes of $X_{ess}$ is the collection of essential hyperplanes of $X$
\item $f$ is bounded-to-one. 
\end{itemize}
\end{theorem}

\noindent{\bf Remark.} A natural subclass of \cat cube complexes is the collection of those which have  \emph{extendible geodesics}, which means that every finite geodesic path can be extended to a bi-infinite geodesic. It is easy to see that such complexes are essential. 

We call $Y$ the essential quotient of $X$. The map $f$ above is easily seen to extend to a bounded-to-one map $\cU(X)\to\cU(X_{ess})$. In the rest of the paper, we will pass to this essential quotient and work with it.

\subsection{Products}
As was discussed in \cite{CS}, a decomposition of $X$ as a product $X=X_1\times X_2$ corresponds to a decomposition of $\cHH$ as a disjoint union $\cHH=\cHH_1\cup\cHH_2$ where every hyperplane in $\cHH_1$ intersects every hyperplane in $\cHH_2$. A \cat cube complex is called \emph{irreducible} if it does not decompose as a product.  We recall the following theorem, proved in \cite{CS}.

\begin{theorem}[Product Decomposition Theorem] Every finite dimensional \cat cube complex admits a canonical (up to permutation of factors) decomposition as a finite product of irreducible 
\cat cube complexes. 
\end{theorem}

We note that the canonical property of this decomposition is that up to passing to a subgroup of finite index, $Aut(X)$ preserves the decomposition.
We now simply observe that an ultrafilter on $X$ gives rise (by restriction) to an ultrafilter on each of the irreducible factors. Conversely, a choice of an ultrafilter on each factor gives us an ultrafilter on $X$. Thus, if $X\cong\prod_{i=1}^n X_i$ is the canonical product decomposition of $X$, we  have an identification $\cU(X)\cong\prod_{i=1}^n \cU(X_i)$.

An unbounded cocompact  \cat cube complex is called \emph{Euclidean} if it contains an $Aut(X)$-invariant flat, otherwise it is called \emph{non-Euclidean}. Only the non-Euclidean factors will play an important role in the description of the boundary. It is thus useful to separate all the Euclidean factors from all the non-Euclidean ones. We summarize this section as follows. 

\begin{corollary}
\label{DeRahmDecomposition}
Let $X$ be an unbounded, proper cocompact  \cat cube complex. Then 
\begin{enumerate}
\item\label{Essential} $X$ admits a bounded-to-one $Aut(X)$-equivariant essential quotient $X_{ess}$  
\item $X_{ess}$ admits an $Aut(X)$-invariant decomposition $X_{ess}=X_P\times X_E$, where $X_E$ is Euclidean and $X_P$ is a product of irreducible non-Euclidean complexes. 
\end{enumerate}
\end{corollary}

Part (1) of the corollary tells us that if we are considering proper cocompact actions on \cat cube complexes, we may pass to actions on essential cube complexes. Note that then each of the factors described in Part (2) of the corollary are also essential. An unbounded essential \cat cube complex whose irreducible factors are all non-Euclidean will be called a \emph{strictly non-Euclidean complex}. 
\subsection{Splitting off the Euclidean factor}

Consider now a group $G$ acting properly and cocompactly on an essential complex $X=X_P\times X_E$. We wish to obtain a corresponding splitting of $G$; that is, we aim to show that uniform lattices in 
$\Aut(X)$ are reducible. Such a result is true in a much more general setting, due to work of  \cite{CM}. In our setting, the matter is simplified by the following fact. 

\begin{lemma}
Let $X$ be a Euclidean complex, then $\Aut(X)$ is discrete.
\end{lemma}

\begin{proof}
Let $\EE\subset X$ be a flat invariant under $\Aut(X)$. By definition $Aut(X)$ acts cocompactly on a Euclidean complex $X$, and it follows that for some $R>0$, the $R$-neighborhood of $\EE$ is $X$. Since $X$ is essential, this tells us that all the hyperplanes cross $\EE$. 

Let $W$ be some bounded open subset of $X$ intersecting $\EE$ nontrivially.  Since $\Aut(X)$ acts cellularly on $X$, there exists some neighborhood $U$ of the identity in $\Aut(X)$ which fixes $W$(pointwise).  It follows that $U$ fixes $\EE$. It follows that $U$ acts trivially on the collection of hyperplanes of $X$. Since both halfspaces defined by a hyperplane contain hyperplanes, it follows that $U$ acts trivially on the halfspaces of $X$. Consequently, $U$ acts trivially on $\cU(X)$ and in particular on the vertices of $X$. We thus have shown that $U$ contains only the identity. 
\end{proof}

A theorem of Caprace and Monod \cite{CM} tells us that an irreducible lattice in a product of two infinite, proper \cat spaces projects to an indiscrete action on each of its factors. We then obtain as a corollary the following.

\begin{cor}
Let $X=X_P\times X_E$ be a \cat cube complex and $G$ a group acting properly and cocompactly on $X$. Then there exists a finite index subgroup $H<G$ such that $H=H_P\times H_E$, where $H_P$ acts properly and cocompactly on $X_P$ and $H_E$ acts properly and cocompactly on $X_E$. 
\end{cor}

\subsection{Flipping, skewering and facing hyperplanes}

We will be using some notions and  results from \cite{CS} which we record here for convenience. 
An automorphism $g\in Aut(X)$ of a \cat cube complex is said to \emph{skewer} a hyperplane $\hh$ if for some $n>0$ and a halfspace $\h$ bounded by $\hh$ we have that $g^n\h\subsetneq \h$. We say that $g$ \emph{flips} a half-space  space $\h$ if $g\h^*\subset\h$. A halfspace for which there exists no $g\in Aut(X)$ which flips it is said to be \emph{unflippable.} We then have the following results. 

\begin{lemma}[Single Skewering]
Let  that $X$ is a finite dimensional \cat cube complex and let  $G$ a group acting properly and cocompactly on $X$. Then for every essential hyperplane $\hh$ in $X$, there exists $g\in G$ such that $g$ skewers $\hh$. 
\end{lemma}

\begin{theorem}[Flipping Lemma]\label{thm:Flipping}
Let $X$ be an unbounded \cat cube complex and let  $G$ be a group acting properly and cocompactly on $X$. Let also  $\h$ be a half-space which is unflippable by the  action of $G$. 

Then $X$ has a decomposition $X=X_1\times X_2$ into a product of subcomplexes, corresponding to a transverse hyperplane decomposition $\cHH(X)=\cHH_1 \cup \cHH_2$, which satisfies the following properties.
\begin{enumerate}
\item $X_1$ is irreducible and all of its hyperplanes are compact. 
\item Some finite index subgroup $G' \leq G$ preserves the decomposition $X=X_1\times X_2$.
\item The $G'$-orbit of $\hh$ is in $\cHH_1$. 
\item $X_1$ is $\RR$-like, namely quasi-isometric to the real line.
\end{enumerate} 
\end{theorem}
 
 \begin{corollary}[Double Skewering]
 Let $X$ be an essential \cat cube complex and $G$ a group acting properly and cocompactly on $X$. Then for every pair of disjoint hyperplanes, there exists a group element skewering both. 
 \end{corollary}
\noindent{\bf Remark.} The statement in \cite{CS} is somewhat more general than what is stated here, but this is sufficient for our needs. 

The next two results involve the existence of hyperplanes which are lie in a particular configuration with respect to one another. 

\begin{proposition}[Corner Lemma]
\label{CornerHyperplane}
Let $X$ be an essential, non-Euclidean, irreducible cocompact \cat cube complex. Let 
$\hh_1$ and $\hh_2$ be two intersecting hyperplanes of $X$. Then there exists a pair of hyperplanes which lie in diagonally opposite components of $X\setminus \{\hh_1,\hh_2\}$. (That is, the hyperplanes are separated both by $\hh_1$ and $\hh_2$). 
\end{proposition}

\begin{proposition}[Facing Triple Lemma]
\label{FacingTriple} 
Let $X$ be an essential non-Euclidean, cocompact complex. Then there exists a facing triple of hyperplanes in $X$. That is, there exists a triple of disjoint hyperplanes no one of which separates the other two. 
\end{proposition}

\section{The definition of the boundary B(X)}

\subsection{Non terminating ultrafilters}

For the rest of this paper we shall usually have as a standing assumption that $X$ is an unbounded, locally finite, finite dimensional \cat cube complex. The focus of this paper will be such complexes whose automorphism group acts cocompactly. A complex which has a cocompact automorphism group will be called a \emph{cocompact complex}. By Corollary \ref{DeRahmDecomposition}, it suffices to study such complexes which are essential. 

Moreover, in light of the fact the Poisson boundary of an abelian group is a point, the Euclidean factor will play no role in the construction of the boundary. 
Thus given an arbitrary group $G$ acting properly and cocompactly on an essential complex $X=X_P\times X_E$, we will consider a finite index subgroup $H$ of $G$ which preserves the decomposition into the strictly non-Euclidean and Euclidean factors, such that $H=H_P\times H_E$, and focus on the action of $H_P$ on $X_P$. The boundary constructed will be the boundary of $G$ as well. This follows from the fact that $H_E$ has an Abelian subgroup of finite index which is central in $H$, and the center of any group acts trivially on its Poisson boundaries.

Thus, we will now let $X$ be a cocompact, essential, strictly non-Euclidean \cat cube complex. 
We shall be interested in a particular subset of the Roller boundary $\ux$.  An ultrafilter which has the property that no descending collection of halfspaces terminates is called {\it nonterminating}. Let $\nonterm(X)$ denote the collection of nonterminating ultrafilters, namely : 

$$\nonterm(X)=\{\alpha\in\ux\vert \h\in\alpha\ \Rightarrow\exists\ \h^\prime\in\alpha\ {\rm with }\ \h^\prime\subsetneq\h\}$$
We let $B(X)$ denote the closure of $\nonterm(X)$ in the Tychonoff topology on $\ux$ defined above. 

We note that if $X$ decomposes as a product $X=\prod_i X_i$, then $\nonterm(X)=\prod_i\nonterm(X_i)$. Consequently, we obtain that  $B(X)=\prod_i B(X_i)$. 


Our aim will be to show that the action on the boundary $B(X)$  of $X$ enjoys the dynamical properties we are interested in. 
The very first thing we need is that the boundary is not empty.

\begin{theorem}\label{non-empty}{\bf Essential complexes have non-empty boundary.}
Let $X$ be an essential, cocompact \cat cube complex. Then  $\nonterm(X)\not=\emptyset$. 
\end{theorem}

\begin{remark}
The above theorem is false without the assumption of cocompactness of the automorphism group. See Figure~\ref{NoBoundary}.

\begin{figure}[h]
\label{NoBoundary}
\includegraphics{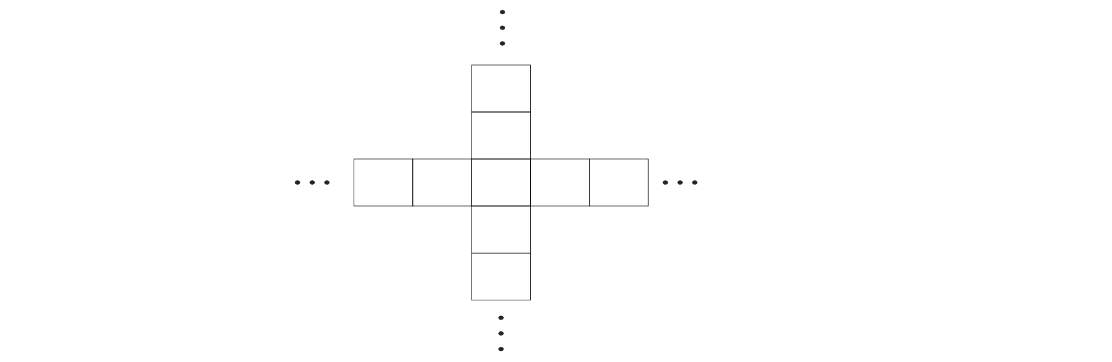}
\caption{An essential complex with no nonterminating ultrafilter.}
\end{figure}
\end{remark}


\begin{proof}

We consider the decomposition of $X=X_P\times X_E$ into factors, where $X_P$ are strictly non-Euclidean and $X_E$ is Euclidean.  Note that by invariance of the decomposition under a finite index subgroup of $Aut(X)$, since $X$ is cocompact so is each of the factors $X_P$ and $X_E$.  Since 
$\nonterm(X)= \nonterm(X_P)\times\nonterm(X_E)$, it suffices to treat separately the cases that $X$ is Euclidean and $X$ is strictly non-Euclidean.

First, let us treat the case that $X$ is Euclidean. The complex $X$ then contains an $Aut(X)$-invariant flat $\EE\subset X$. Since $X$ is cocompact and each hyperplane is essential, it follows that every hyperplane meets $\EE$. The intersection of the hyperplanes of $X$ with $\EE$ gives rise to a collection of lines, which by the finite dimensionality of $X$ fall into finitely many parallelism classes. Since $X$ is essential, each such line has infinitely many parallel lines on either side of it.  Choose a ray $r\subset \EE$ which is not parallel to any of these lines. For each hyperplane, choose the halfspace containing the infinite part of $r$. This is then a nonterminating ultrafilter.

For the strictly non-Euclidean factor $X_P$, we first note that $X_P$ is itself of a product of non-Eucliean irreducible factors $X_P=\prod X_i$. Since then $\nonterm(X_P)=\prod_i\nonterm(X_i)$, it thus suffices to consider the case that is that $X$ is an essential, irreducible, non-Euclidean complex. We then proceed as follows.

We let $\cH=\{\hh_1,\hh_2,\ldots\}$ be some ordering of the hyperplanes of $X$. Recall that an ultrafilter $\alpha$ is a choice of a halfpace $\h$ for each hyperplane $\hh$ satisfying the consistency condition: $\h\in\alpha$ and $\h\subset\k$ then $\k$ in $\alpha$.  We will construct an ultrafilter $\alpha$ by describing the halfspaces $\h_i\in\alpha$ associated to $\hh_i$ in order. 

Choose $\h_1\in\alpha$ arbitrarily as a halfspace bounded by $\hh_1$. For every hyperplane $\hh\subset \h_1^*$, we choose the halfspace $\h$ such that $\h_1\subset\h$ (this is dictated by the consistency condition.)

Now we consider the hyperplane $\hh_n$ such that $n$ is the smallest number for which $\hh_n\not\subset\h_1^*$. That is, the first $n$ for which the choice of the halfspace $\h_n$ has yet to be made.

If $\hh_n\cap\hh_1=\emptyset$, then we choose $\h_n$ so that $\h_n\subset\h_1$. If $\hh_n\cap\hh_1\not=\emptyset$, then by Lemma \ref{CornerHyperplane}, there exists a hyperplane $\hh_m$ such that $\hh_m\subset\h_1$ and $\hh_m\cap\hh_n=\emptyset$. Let $\h_m$ denote the halfspace bounded by $\hh_m$ not containing $\hh_1$ and $\hh_n$ and let $\h_n$ denote the halfpspace bounded $\hh_n$ which contains $\hh_m$. Now for every hyperplane $\hh\subset \h_m^*$, we choose the haflspace $\h$ such that $\h_m\subset\h$. Note that these choices do not change whatever choices were made previously. Note that we have now made choices for $\hh_1$ and $\hh_n$. We have thus arranged that there exists a halfspace  $\h\in\alpha$ such that $\h\subset\h_k$, for $k=1,\ldots,n$. 

We continue in this manner choosing a halfspace for each hyperplane $\hh$. Note that for any two hyperplanes $\hh$ and $\hk$, the decisions for both are made at some finite stage so that the consistency condition for an ultrafilter are satisfied (i.e. if $\h\in\alpha$ and $\h\subset\k$ then $\k\in\alpha$. Also, by construction, for each $\h\in\alpha$, there exists $\k\in\alpha$ such that $\k\subset\h$. Thus,  $\alpha$ is nonterminating, as required. 
\end{proof}

\begin{remark}
Note that the above argument shows  more, namely that there exists an element of $\nonterm(X)$ in each halfspace of $X$. This fact will be used below.
\label{ManyIDU}
\end{remark}

We now show that the nonterminating ultrafilters make up ``most" of $B(X)$. 

\begin{proposition}
\label{DenseGDelta}
Let $X$ be an essential, strictly non-Euclidean, cocompact \cat cube complex. Then $\nonterm(X)$ is a dense $G_\delta$ in $B(X)$.
\end{proposition}

\begin{proof}

For a given half-space $\h$ in $X$ we let

$$B_\h(X)=\{\alpha\in B(X)\vert \h \text{ is minimal in } \alpha\}$$
Note that $$B(X)-\nonterm(X)=\bigcup_{\h} B_\h(X)$$ 
Thus, to show that $\nonterm(X)$ is a dense $G_\delta$ in $B(X)$, it suffices to show that each $B_\h(X)$ is closed and has empty interior. 

Note that $B_\h(X)\cap\nonterm(X)$ is empty and $B(X)$ is defined to be the closure of $\nonterm(X)$ so that $B_\h(X)$ has empty interior. To see that $B_\h(X)$ is closed, simply observe that  

$$B(X)- B_\h(X)=(U_{\h^*} \cup \bigcup_{\k\subsetneq\h} U_\k) \cap B(X)$$

Thus, the complement of $B_\h(X)$ is open, as required. 
\end{proof}


\section{General boundary theory}

Let $B$ be a compact metrizable space, and let $G$ be a group of homeomorphisms of $B$.
We recall some general definitions related to the dynamics of the $G$-action on $B$, and 
then present a general approach to boundary theory. Our approach is motivated by that of  Margulis \cite[Ch. VI]{M} and Guivarc'h-Le Page \cite{GLP}, and we will show that it can handle not only linear groups acting on projective space, but is in fact well-suited to handle the boundary theory of cube
 complexes as well.

\begin{definition}\label{MPC}{\bf Minimality, proximality, contractibility} 
\begin{enumerate}

\item The $G$-action on $B$ is called {\it minimal} if for every $b\in B$ and  
 for any non-empty open set $U\subset B$, there exist $g\in G$ such that $gb\in U$. 
\item The $G$-action on $B$ is called proximal if for any two points $b_1,b_2\in B$, there exists a point $c\in B$, such that for every neighbourhood 
$U$ of $c$, there exists $g\in G$ such that $gb_1\in U$ and $gb_2\in U$.
\item A neighbourhood $V$ of a point $b\in B$ is called contractible \cite[Ch. VI, \S 1]{M},  if 
 there exists a point $c\in B$ such that 
for every neighbourhood $U$ of $c$, there exists $g\in G$ such that $gV\subset U$.     
\end{enumerate}
\end{definition}

We shall also make use of a notion which is  stronger than the above three conditions. Roughly speaking, it says that the  ``attracting" conditions stated above for points of $B$ hold with respect to the action on some larger space in which $B$ is contained. 

\begin{definition}\label{LimitSet}
Let  $Y$ be a $G$-space. A subset $B\subset Y$ is said to be a  {\it boundary limit set} for the action of $G$ on $Y$  if $B$ is  $G$-invariant and
\begin{enumerate}
\item for every point $y\in Y$, and every non-empty open set $U\subset Y$ such that $U\cap B\not=\emptyset$, there exists $g\in G$ such that $gy\in U$;
\item for any two points $y_1,y_2\in Y$, there exists a point $c\in B$, such that for every neighbourhood 
$U$ of $c$, there exists $g\in G$ such that $gy_1, gy_2\in U$;
\item for every $y\in Y$ there exist neighborhood $V$ of $y$ and a point $c\in B$, such that  for  every neighbourhood $U$ of $c$, there exists $g\in G$ with $gV\subset U$.     
\end{enumerate}
\end{definition}
We now recall the following definition (see \cite[Ch. VI. 2.13]{M})

\begin{definition}\label{EP}{\bf Equicontinuous decomposition.}
The $G$-action on $B$ is said to admit an equicontinuous decomposition if we can write 
$G$ as a finite union of subsets $G=\cup_{i=1}^N G_i$, and find non-empty open subsets $B_i\subset B$ 
satisfying $G\cdot B_i=B$, such that the set of functions $G_i$ (from $B_i$ to $B$) is equicontinuous 
on $B_i$.  Equivalently, fixing a metric $d$ on $B$, for every $\epsilon > 0$ there exists $\delta > 0$ such that 
if $b,b^\prime\in B_i$ satisfy $d(b,b^\prime)< \delta$ then $d(gb,gb^\prime)< \epsilon$, for all 
$g\in G_i$ simultaneously.     
\end{definition}

We now consider the space $P(B)$ of Borel probability measures on $B$, which we take with the 
(metrizable, separable) $w^\ast$-topology, namely $\eta_n\to \eta$ if and only if for every continuous function $f$ on $B$, $\int_Bfd\eta_n\to \int_Bf d\eta$.  For each $b\in B$, let $\delta_b$ denote the point  measure supported on $b$. The map $\b\mapsto\delta_b$ is continuous and embeds $B$ as a compact subset of $P(B)$, the subset of point measures.  

Given a probability measure $\mu$ on $G$, we 
can consider the convolution $\mu\ast\eta=\int_G g\eta d\mu(g)$, which is another probability measure on $B$. The operator $\eta\mapsto \mu\ast\eta$ is continuous and admits a fixed point $\nu$ satisfying 
$\mu\ast\nu=\nu$. Such a measure is called a $\mu$-stationary measure on $B$.  If $\nu$ is a $\mu$-stationary measure, the pair $(M,\nu)$ is called a $(G,\mu)$-space. 

The following two definitions were introduced in \cite{poisson}. 

\begin{definition}\label{boundary}{\bf Strong proximality, boundary space}
\begin{enumerate}
\item  The $G$-action on $B$ is called strongly proximal if given any probablity measure $\nu\in P(B)$, 
there exists $b\in B$ and a sequence $g_n\in G$ such that $g_n\mu\to \delta_b$. 
\item A minimal strongly proximal $G$-space $B$ is called a boundary. 
\end{enumerate}
\end{definition} 

A statistical version of proximality is given by the following 

\begin{definition}\label{mu-boundary}{\bf $\mu$-boundary, $\mu$-proximality, mean proximality}. 
\begin{enumerate}
\item The $(G,\mu)$-space $(B,\nu)$ is called a $\mu$-boundary if given a sequence $\omega=(\omega_n)_{n\in \NN}$ of independent  random variables with values in $G$ and common distribution $\mu$, the sequence of probability measures $\omega_1\omega_2\cdots\omega_n \nu$ converges 
with probability one to a limit point measure $\delta_{Z(\omega)}$. 
\item The $G$-space $B$ is called $\mu$-proximal if for every $\mu$-stationary measure 
$\nu\in P(B)$, the $(G,\mu)$-space $(B,\nu)$ is a $\mu$-boundary. 
\end{enumerate}
\end{definition}

We recall the following  results. 

\begin{theorem}\label{M1}\cite[Ch VI. Prop. 1.6]{M}
\begin{enumerate}
\item
If $B$ is a proximal $G$-space and every point has a contractible neighbourhood, then 
$B$ is strongly proximal.
\item If $B$ is proximal, minimal and contains a non-empty  contractible open set, then it is strongly proximal, namely a boundary of $G$.
\end{enumerate}
\end{theorem}

\begin{theorem}\label{GR}\cite{GLP} 
If the $G$-action on $B$ is strongly proximal and admits an equicontinuous decomposition, then for every 
probability measure $\mu$ on $G$ whose support generates $G$ as a semigroup, 
the $\mu$-stationary measure $\nu$ is unique, and 
$(B,\nu)$ is a $(G,\mu)$-boundary. 
\end{theorem}

   
 We summarize the previous discussion in the following 
   
\begin{corollary}\label{conditions}
If the $G$-action on $B$ is minimal, proximal, contains a non-empty contractible open set and admits 
an equicontinuous decomposition, then  $(B,\nu)$ is  a boundary.  In addition, it is a $(G,\mu)$-boundary for the unique $\mu$-stationary 
measure $\nu$ on $B$, where $\mu$ is any probability measure on $G$ whose support generates 
$G$ as a semigroup.
\end{corollary}

Corollary \ref{conditions} will be our basic tool : we will prove below that the action of a subgroup $G$ of the automorphism group of a \cat cube complex satisfes the hypotheses required 
by the corollary, provided certain natural assumptions on the cube complex and the group $G$ are satisfied. 


\section{Establishing that $B(X)$ is a boundary}

The aim of this section is to show that for an essential, strictly non-Euclidean \cat cube complex $B(X)$ is a boundary. More precisely, we establish the following theorem.

\begin{theorem}\label{B(X) UMSP}
Let $X$ be an essential, strictly non-Euclidean \cat cube complex  admitting a proper co-compact action 
of  $G\subset \Aut(X)$. Then the $G$-action on $B(X)$ is minimal and strongly proximal, and every probability measure $\mu$ whose support generates as a semigroup $G$ has a unique stationary measure $\nu$. 
The $(G,\mu)$-space $(B(X),\nu)$ is a $\mu$-boundary.  
\end{theorem}

In order to prove Theorem \ref{B(X) UMSP}, by Corollary \ref{conditions} it suffices to show that the action satisfies conditions laid out there. In particular, we will show that $B(X)$ is a boundary limit set and that it admits an equicontinuous decomposition.  

\subsection{$B(X)$ is a boundary limit set}

 In order to show that $B(X)$ is a boundary limit set, we need  
to explore the  topology on $\ux$ and $B(X)$ a bit more closely. Recall  that the collection of subsets of the form $U_\h=\{\alpha\vert \h\in\alpha\}$, called the collection of $\ux$-halfspaes,   forms a sub-basis for the Tychonoff topology on $\ux$, so that a basic open set consists of the intersection of finitely many $\ux$-halfspaces. Since $B(X)$ is a subset of $\ux$, we obtain a basis of open sets for $B(X)$. 


We now observe that for points in $\nonterm(X)$ the local neighborhood basis has additional structure, and this fact will play a crucial role below. 
First, we introduce the following definition. 

\begin{definition}{\bf Sectors.} 
\begin{enumerate}
\item A sector $\s$ in a cube complex $X$ is a finite intersection of halfspaces $\cap_{i=1}^n \h_i$, for which the corresponding hyperplanes $\hh_1,\ldots,\hh_n$ meet, namely $\bigcap_{i=1}^n \hh_i\neq \emptyset$. 
\item A sector $S$ of $\ux$ is a finite  intersection of $\ux$-halfspaces of the form $\bigcap_{i=1}^nU_{\h_i}$ for which the corresponding hyperplanes $\hh_1,\ldots,\hh_n$  meet, namely 
for which $\cap_{i=1}^n \h_i$ constitute a sector in $X$. 
\end{enumerate}
When $S$ is a sector $\ux$, we will use $\s$ to denote the corresponding sector in $X$ and visa versa. 
\end{definition}

We then have the following result. 

\begin{lemma}\label{SectorBasis}
Let $X$ be a  \cat cube complex. Then every neighborhood of a point $\alpha \in\nonterm (X)$ contains a neighborhood of $\alpha$ which is a sector in $\ux$.  
\end{lemma}

\begin{proof}

First we define a certain notion of complexity for basic open sets. Given a basic open set 
$U=\cap_{i=1}^n U_{\h_i}$, we define two numbers, $D(U)$ and $N(U)$. The number $N(U)=n$ is simply the number of halfspaces defining $U$. The number $D(U)$ is the sum of the distances between the hyperplanes (minimized over all presentations of $U$ as an intersection of half-spaces):

$$D(U) = \sum d(\hh_i, \hh_j) ,$$
where the distance $d$ between non-intersecting  hyperplanes is the minimal length of a 1-skeleton path  crossing both hyperplanes, minus one.  When $\hh$ and $\hh^\prime$ intersect, $d(\hh,\hh^\prime)=0$. 
We set $C(U)=(N(U), D(U))$, and get a partial ordering on the basic open sets by ordering $C(U)$ lexicographically.

Suppose that $U=\cap_{i=1}^n U_{\h_i}$ is a basic open set containing $\alpha\in\nonterm(X)$ so that $\alpha$ is nonterminating. Now we may assume that $C(U)$ is minimal amongst all basic open sets contained in $U$ and containing $\alpha$. 

If $N(U)=1$ or $D(U)=0$, then $U$ is a sector and we are done. So suppose that $N(U)>1$ and $D(U)>0$.

First, suppose that for some $i,j$, we have $\h_i\subset\h_j$. If this were the case, then $U$ is the intersection of $\{\h_1,\ldots\h_n\} \setminus \{\h_j\}$ and so $N(U)$ can be reduced. So we may assume that for all $i\neq j$, $\h_i\not\subset\h_j$.  We call the process just described ``removing extraneous halfspaces".

Now after a possible renumbering of the $\h_i$'s, the associated  hyperplanes $\hh_1,\hh_2$ are  disjoint and $d(\hh_1,\hh_2)>0$. Now since $\alpha$ is nonterminating and $\h_1\in\alpha$, there exists a halfspace $\h_1^\prime\in\alpha$ such that $\h_1^\prime\subset \h_1$.  We will now let 

 $$U^\prime=U_{\h_1}^\prime\cap\left(\cap_{i=2}^n U_{h_i}\right).$$

By construction, we have that $U'\subset U$ and $\alpha\in U'$. Thus, by the minimality of $C(U)$, we may assume that $C(U)\leq C(U')$.

For each $j=2,\ldots,n$, suppose that $\hh_1^\prime$ is disjoint from $\hh_j$ and does not separate $\hh_1$ from $\hh_j$, in which case we obtain that $\h_1^\prime\subset\h_j$.  Thus $U^\prime$ has extraneous halfspaces that can be removed, reducing $N(U^\prime)$ and hence $C(U^\prime)<C(U)$, a contradiction. 

We thus have two possibilities.
\begin{itemize}
\item $\hh_1^\prime$ intersects every $\hh_j$, $j=2,\ldots,n$. In this case $D(U')<D(U)$, a contradiction.

\item There exists $\hh_j$, disjoint from $\hh_1^\prime$. In this case, for each such $j$, $\hh_1^\prime$ must separate $\hh_1$ and $\hh_j$, so that $d(\hh_1^\prime,\hh_j)<d(\hh_1,\hh_j)$. We thus  obtain $D(U')<D(U)$, again a contradiction.

\end{itemize}
\end{proof}

We will now need a variant of the Corner Lemma which ensures that certain types of halfspace intersections contain hyperplanes.  Let $\h_1$ and $\h_2$ be two halfspaces associated to intersecting hyperplanes  $\hh_1$ and $\hh_2$. Note that the hyperplane $\hh_1$ separates the hyperplane $\hh_2$ into two \emph{half-hyperplanes}, $\hh_2^+=\hh_2\cap\h_1$ and $\hh_2^-=\hh_2\cap\h_1^*$.  Similarly, $\hh_1$ is subdivided into two half-hyperplanes by $\hh_2$, where $\hh_1^+=\hh_1\cap\h_2$.  A half-hyperplane is called \emph{$R$-shallow} if it contained in an $R$-neighborhood of $\hh_1\cap\hh_2 $. Otherwise it is called \emph{deep}. Note that $\h_1\cap\h_2$ is bounded by the union of the two half-hyperplanes $\hh_1^+$ and $\hh_2^+$. Note also that since there are only finitely many orbits of half-hyperplanes, there exists some global constant $R>0$ such that if $\hh_1^+$ is not $R$-shallow, then $\hh_1^+$ is deep. We record this fact.

\begin{remark}
If $X$ is a cocompact cube complex, then there exists a number $R>0$ such that any half-hyperplane which is not $R$-shallow is deep.
\label{RShallow}
\end{remark}

\begin{lemma}
Let $X$ be a cocompact, essential, irreducible, non-Euclidean \cat cube complex  and let $\hh_1$ and $\hh_2$ be hyperplanes in $X$.
Let $\h_1\cap\h_2$ be one of the four sectors defined by intersecting hyperplanes $\hh_1$ and $\hh_2$. Suppose that one of the half-hyperplanes $\hh_1^+$ and $\hh_2^+$ is deep. Then the sector $\h_1\cap\h_2$ contains a hyperplane.
\label{Corner}
\end{lemma}

\begin{proof} There are several cases, which we handle by order of difficulty.

\medskip
\noindent{\bf Case 1: One of $\hh_1^+, \hh_2^+$ is shallow.} Suppose without loss of generality, the $\hh_1^+$ is shallow, so that $\hh_2^+$ is deep. Thus, there exists $R>0$ such that $\hh_1^+$ is contained in the $R$-neighborhood of $\hh_2$. Since $\hh_2$ is essential, there exists a hyperplane $\hk\subset\h_2$ such that $d(\hk,\hh_2)>R$. It follows that $\hk\cap\hh_1=\emptyset$. 
If $\hk\subset\h_1$, we are done, so suppose that $\hk\subset \h_1^*$. Consider a geodesic path 
$\gamma$ joining $\hk$ and $\hh_2$ and let $p=\gamma\cap\hh_2$. Since the action of $stab(\hh_2)$ on $\hh_2$ is cocompact and $\hh_2^+$ is deep, there exists some $g\in\stab(\hh_2)$ such that $gp\in\h_1$. We then have that $g\gamma\in\h_1$, so that $g\hk\in\h_1$, as required.

\noindent{\bf Case 2: Both $\hh_1^+, \hh_2^+$ are deep and one of $\hh_1^-, \hh_2^-$ is shallow.}
Suppose without loss of generality that $\hh_1^-$ is shallow. By applying case 1, we have that there exists a hyperplane $\hk\subset\h_1\cap\h_2^*$.  Since the action of $\stab(\hh_1)$ on $\hh_1$ is cocompact, there exists some $g\in\stab(\hh_1)$ such that $g\hh_2\subset\h_2$. Since $\hh_1^-$ is shallow, we know that $g$ does not skewer $\hh_2$. It follows that $g\h_2^*\subset\h_2$. Thus $\hk\subset\h_1\cap\h_2$, as required.

\noindent{\bf Case 3: All half hyperplanes $\hh_1^\pm,\hh_2^\pm$ are deep.}  By Lemma \ref{CornerHyperplane}, there exists a pair of diagonally opposite sectors of $\hh_1$ and $\hh_2$ which contain hyperplanes. If that pair is $\h_1\cap\h_2$ and $\h_1^*\cap\h_2^*$, then we are done, so we may assume that $\h_1\cap\h_2^*$ and $\h_1^*\cap\h_2$ both contain hyperplanes.  
If $\hh_2\cap\hh_1$ is flippable in $\hh_1$ (i.e. there exists $g\in\stab(\hh_1)$ such that 
$g\h^*_2\subset\h_2$), then letting $\hk$ denote a hyperplane contained in $\h_1\cap\h_2^*$, we would have that $g\hk\subset\h_1\cap\h_2$, as required.

So let us assume that $\hh_2\cap\hh_1$ is unflippable in $\hh_1$. We let $\hk$ denote a hyperplane such that $\hk\subset\h_1^*\cap\h_2$. Since the hyperplane $\hh_1\cap\hh_2$ is essential in $\hh_2$, there exists $g\in\stab(\hh_2)$ such that $g\h_1\subset\h_1$.  We then have that for some $n$, either $g^n\hk\subset\h_1$ or $g^n\hk\cap\hh_1\not=\emptyset.$ By passing 
to possibly some larger power of $g$, we obtain that the half-hyperplane $g^n\hk^+=g\hk\cap\h_1$ is deep.  

Let $\hm=g^n\hk$. We now observe that for some power of $n>0$, $g^n\hm\cap\hm=\emptyset$. This is clear since the axis of $g$ does not lie in neighborhood of $\hm$.  There are three hyperplanes in $\hh_1$: $L_1=\hh_1\cap\hh_2$, $L_2=\hh_1\cap\hm$ and $L_3=\hh_1\cap g^n\hm$. We already are assuming that $L_1$ is essential in $\hh_1$. If either of $L_2$ or $L_3$ is inessential in $\hh_1$, we appeal to Case 1 above to produce the desired hyperplane. It follows that both $L_2$ and $L_3$ are essential in $\hh_1$. Thus, by the Flipping Lemma \ref{thm:Flipping}, the hyperplanes $L_1$, $L_2$ and $L_3$ do not form a facing triple. This means that one of them, without loss of generality $L_2$ separates $L_1$ from $L_3$. But this contradicts the fact that  $\hh_2$, $\hm$ and $g^n\hm$ form a facing triple in $X$. 
\end{proof}

Now we can prove our central technical result. 
\begin{proposition}
Let $X$ be a cocompact, essential, irreducible, non-Euclidean \cat cube complex admitting a proper cocompact group action. 
Let $S=\cap_{i=1}^k U_{\h_i}$ be a sector neighborhood of a nonterminating ultrafilter. Then there exists a hyperplane $\hh\in\s$. 
\label{SectorHyperplane}
\end{proposition}

\begin{proof} The proof will be by induction on $n$. We start with $n=2$.  Let $\s=\h_1\cap\h_2$. If either of the half-hyperplanes bounding $\s$ are deep, then we may apply Lemma \ref{Corner} to conclude that $\s$ contains a hyperplane. Otherwise, suppose that $\hh_1^+$ is shallow, then we may apply the fact that $\alpha$ is nonterminating to construct a sequence of halfspaces 
$\h_2\supset\h_3\ldots$ such that $\alpha\in U_{\h_i}$ for all $i>1$. Now since $\hh_1^+$ is shallow there exists some $n$ such that $\hh_n\cap\hh_1=\emptyset.$ We then have that $\hh_n\subset\h_1\cap\h_2$ as required. 

We now suppose that $n>3$. Let $L=\bigcap_{i=2}^n \hh_i$.  Let $L^+=L\cap\h_1$ and $L^-=L\cap\h^*$. We first claim that $L^+$ is not contained in any neighborhood of $\hh_1$.  Since 
$\alpha\in U_{\h_1}$ and $\alpha$ is nonterminating, there exists a sequence of halfspaces 
$\h_1\supset\m_2\supset\m_3\ldots$ such that $\alpha\in U_{\m_i}$ for all $i$. Suppose that for some $i$ we have  $\hm_i\cap L=\emptyset$. Then there exists some $\hh_i$, $2\leq i\leq n$ such that $\hh_i\cap\hm=\emptyset$. Let $F\subset\{2,\ldots,n\}$ such that $\hh_i\cap\hm\not=\emptyset$. 

By assumption, we have that $F$ is a proper subset of $\{2,\ldots,n\}$. If $F=\emptyset$, then we would have that $\hm\subset \s$ and we are done. So now let $\s^\prime=\m\cap\bigcap_{i\in F} \h_i$. By construction $\s^\prime\subset\s$ and $\alpha\in \s^\prime$. Thus, by induction,  $\s^\prime$ contains a hyperplane and we are done. We thus have that all the hyperplanes $\hm_i$ intersect $L^+$. Since the hyperplanes $\hm_i$ correspond to a nested collection of halfspaces, we have shown our claim. 

Next we apply the inductive hypothesis to the sector $\bigcap_{i=2}^n\h_i$ to conclude that there exists a hyperplane $\hk\in\bigcap_{i=2}^n\h_i$. If $\hk\subset\h_1$, we are done, so suppose not, so that $\hk\subset\h_1^*$ or $\hk\cap\hh_1=\emptyset$. 

Let $R$ denote the number described in Remark \ref{RShallow}, so that any half-hyperplane which is not $R$-shallow is deep. 
We claim that there exists  $g\in\stab(L)$ such that $g\hk\cap\h_1$ contains a point at distance greater than $R$ from $\hh_1$. To see this, let $p$ be some point of $\bigcap_{i=1}^n \hh_i$. Let $q$ be a point of $\hk$ and let $D=d(p,q)$. Now since $L^+$ is deep, there exists $g\in\stab(\L)$ such that $d(gp,\hh_1)>R+D$. It follows from the triangle inequality that $d(gq,\hh_1)>R$. 

From this claim it follows that either $g\hk\subset\h_1$, in which case we are done, or $g\hk\cap\hh_1\not=\emptyset$ and the half-hyperplane $g\hk^+=g\hk\cap\h_1$ is deep. In the latter case,
we apply Lemma \ref{Corner} to conclude that both sectors bounded by $g\hk^+$ and $\hh_1$ contain hyperplanes. One of these sectors is contained in $\s$, so that we have a hyperplane contained in $\s$ as required. 
\end{proof}

Putting together the previous two lemmas, we obtain the following. 
\begin{corollary}
Let $X$ be a cocompact, essential, irreducible, non-Euclidean \cat cube complex.
If $U$ is an open set in $\cU(X)$ such that $U\cap B\not=\emptyset$, then $U$ contains a sub-basic open set $U_\h\subset U$.
\label{HypsInOpenSets}
\end{corollary}
We are now ready to prove that $B(X)$ is indeed a boundary limit set.

\begin{theorem}\label{B(X) limit set}
Let $X$ be an essential, strictly non-Euclidean \cat cube complex  admitting a proper co-compact action of  $G\subset \Aut(X)$.
Then $B(X)$ is a boundary limit set for the $G$ action on $\ux$. 
\end{theorem}

\begin{proof}[Proof of Theorem  \ref{B(X) limit set}] We need to establish the three properties in 
Definition \ref{LimitSet} where $Y=\cU(X)$ and $B=B(X)$.

We will prove the following stronger claim which immediately implies all three properties.

\medskip
\noindent{\bf Claim.} \emph{For any two ultrafilters $\alpha$ and $\beta$ there exists an open neighborhood $V$ of $\alpha$ and $\beta$ such that for any open set $U\subset \cU(X)$ such that $U\cap B(X)\not=\emptyset$, there exists $g\in G$ such that $gV \subset U$.}
\medskip

First, we establish the claim in the event that $X$ is irreducible. Let $\alpha,\beta\in \cU(X)$. By the Facing Triple Lemma \ref{FacingTriple, there exists in $X$ a facing triple of hyperplanes $T$.}  It follows that there exists a halfspace $\h$, associated to a hyperplane in $T$, such that $\alpha, \beta\in U_\h$.  We set $V=U_\h$. Now given any open set $U$ such that
 $U\cap B(X)\not=\emptyset$, we know by Corollary  \ref{HypsInOpenSets} that there exists a hyperplane $\hk\subset U$.
 Thus, for one of the halfspaces bounded by $\hk$, we have $U_\k\subset U$. Now by appealing to the Flipping Lemma or Double Skewering Lemma of \cite{CS}, we have $g\in G$ such that 
 $gV\subset U_\k\subset U$, as required. 
 
 We now need to establish the claim when $X$ is not irreducible. We shall establish the claim when $X$ is a product of two irreducible factors. The general case is established by induction. 

Suppose that $X=X_1\times X_2$. After possible passing to a subgroup of finite index in $G$, we may assume that $G$ preserves the decomposition.  Let $\alpha=(\alpha_1,\alpha_2)$ and
$\beta=(\beta_1,\beta_2)$ be two ultrafilters in $\cU(X)$. As above, in each factor $X_i$, we may find a halfspace $\h_i$ such that $\alpha_i,\beta_i\subset U_{\h_i}$. We let $V=U_{\h_1}\times U_{\h_2}$.

Let $U$ be some open set intersecting $B(X)$, so that $U$ contains a nonterminating ultrafilter $\gamma=(\gamma_1,\gamma_2)$ in 
$\nonterm(X)$. It follows that there exists an elementary neighborhood $U_1\times U_2$ such that $\gamma\in U_1\times U_2\subset U$. Note that $\gamma_1$ and $\gamma_2$ are nonterminating so that $U_1$ and $U_2$ intersect $B(X_1)$ and $B(X_2)$ respectively. By Corollary  \ref{HypsInOpenSets}, we have sub-basic open sets $U_{\k_1}\subset U_1$ and $U_{\k_2}\subset U_2$. 

We now choose a vertex $v\in X_2^{(0)}$ and let $X_1^v\equiv X_1\times \{v\}$. Note that there is a natural isometric identification of $X_1$ with  $X^v_1$; we thus identify subsets of $X_1$ with subsets of $X_1^v$.  Note that $\stab(X_1^v)$ acts cocompactly on $X_1^v$. (This is because $G$ acts preserving the product structure so that every element $g\in G$, either  $g\in\stab(X_1^v)$ 
or $g(X_1^v)\cap X_1^v =\emptyset$.) Thus we may apply the claim in the irreducible case to find some $g_1\in\stab(X_1^v)$ such that $g_1U_{\h_1}\subset U_{\k_1}$. 

Next we consider the $\hm\equiv\hk_1\times X_2$ which is a hyperplane in $X$. Note again that $\stab(\hm)$ acts on $\hm$ coccompactly. Choose some vertex in $w\in\hm$ and let $X_2^w\equiv\{w\}\times X_2$. Now $\stab(X_2^w)<\stab(\hm)$ acts cocompactly on $X_2^w$. 
Once again, we may apply the claim in the irreducible case, to find a $g_2\in\stab(X_2^w)$ such that $g_2(U_{\h_2})\subset U_{\k_2}$. Note that $g_2\in\stab(\hm)$ so that it preserves the halfspace $\k_1\times X_2$. We then have that $g_2g_1V\subset U$, as required. 
\end{proof}

\subsection{$B(X)$ admits an equicontinuous decomposition}

\begin{theorem}\label{equi}
Let $X$ be a locally finite essential strictly non-Euclidean cocompact \cat cube complex. Then $B(X)$ admits an equicontinuous decomposition.

\end{theorem} 

We will  need to understand how maximal sectors are nested within one another.
Let $X$ be a finite-dimensional \cat cube complex. Let $\sigma,\tau$ be  maximal cubes with barycenters $\bar{\sigma}$ and $\bar{\tau}$. If $dim(\sigma)=n$ and $dim(\tau)=m$, then $\bar{\sigma}$ is the intersection of $n$ hyperplanes and $\bar{\tau}$ is the intersection of $m$ hyperplanes. There may be some hyperplanes in common. So let
$\hat\cH_\sigma=\{\hh_1,\ldots,\hh_n\}$ be the hyperplanes intersecting $\sigma$ and 
$\hat\cH_\tau=\{\hh_1,\ldots\hh_l,\hk_{l+1},\ldots,\hk_m\}$ be the hyperplanes intersecting $\tau$. Here we assume $\hk_i\not=\hh_j$ for any $i,j$. 

We now note the following about this situation.

\begin{lemma}\
\begin{enumerate}
\item There exists a sector $\s$ defined by $\hat\cH_\sigma$ which contains a sector  $\t$ bounded by $\hat\cH_\tau$.
\item For every hyperplane $\hh$ meeting $\t$, the distance $d(\hh, \bar{\tau})\leq d(\hh,\bar{\sigma})$
\end{enumerate}
\label{SectorContainment}
\end{lemma}

\begin{proof}
For each $1\leq i\leq l$, we choose $\h_i$ arbitrarily. For $l<i\leq l+1$, we know that $\bar{\tau}\not\in\hh_i$, so we let $\h_i$ denote the halfspace containing $\bar{\tau}$.  For every $\hk_i$, $i=l+1,\ldots,m$, we note that $\bar{\sigma}\not\in\hk_i$. Thus we let $\k_i^*$ denote the halfspace not containing $\bar{\sigma}$. We let $\s=\bigcap_{i=1}^n \h_i$ and we let $\t=\bigcap_{i=1}^l\h_i\cap\bigcap_{i=l+1}^m \k_i$. Now it is clear that no hyperplane of $\hat\cH_\sigma$ meets $\t$. For $i=1,\ldots,l$, the hyperplanes of $\hh_i$ bound $\t$ and hence do not intersect it. For $i>l$, $\hh_i$ is disjoint from some $\hk_j$, and by our choice, $\hk_j$ then separates $\hh_i$ from $\t$. 
This completes the proof of $(1)$. 

To prove $(2)$, note that any hyperplane $\hh$ which separates $\bar{\sigma}$ and $\bar{\tau}$ intersects the half space $\k_i^*$ for $i=l+1,\ldots,m$. Also, if $\hh$ also meets $\t$, then it would have to intersect $\k_i$ as well for every such $i$. Note also, that each such $\hh$ would have to also meet $\hh_i$ for $i=1,\ldots,l$, since $\bar{\tau}, \bar{\sigma}\in\hh_i$ for such $i$. 
Thus every hyperplane meeting $\t$ does not separate $\bar{\tau}$ and $\bar{\sigma}$.

Now suppose that $\hh$ is some hyperplane meeting $\t$. Let $\hm_1,\ldots\hm_k$ be the hyperplanes separating $\hh$ from $\bar{\tau}$. These hyperplanes also meet $\t$ and hence do not separate $\bar{\tau}$ from $\bar{\sigma}$. It follows that the hyperplanes $\hm_i$ must be crossed by any path to from $\hh$ to $\bar{\sigma}$. Thus $d(\hh,\bar{\tau})\leq d(\hh, \bar{\sigma})$.
\end{proof}




\begin{proof}[Proof of Theorem \ref{equi}]

We choose a maximal cube $\sigma$ and let $o$ be one of the vertices of $\sigma$. We consider the hyperplanes $\hh_1,\ldots,\hh_n$ which meet $\sigma$. The sectors in $\ux$ defined by the sectors associated with the maximal cube $\sigma$  provide a partition of $\cU(X)$ and consequently of $B$. We denote this partition of $B$ (some of whose sets may be empty) by $B=\cup_{i=1}^{2^n} B_i$.

Now consider a group element $g\in G$ and consider the cube $g\sigma$. By Lemma \ref{SectorContainment}, we have that there exists a pair of sector neighborhoods in $\ux$ $S_i, S_j$ defined by $\sigma$ such that $gS_i\subset S_j$. We define

$$G_{ij}=\{g\in G\vert g(B_i)\subset  B_j\}$$

This will be the requisite decomposition of $G$. More precisely, we will argue that for any $g\in G_{ij}$ and any $\alpha,\beta\in B_i$, $d(g\alpha,g\beta)\leq d(\alpha,\beta)$.  Recall that in the metric on $B$, we order the hyperplanes by their distance from $o$, and seek the first hyperplane that separates $\alpha$ and $\beta$. More precisely, we order the hyperplanes of $X$, $\cH=\{\hh_1,\hh_2,\ldots\}$, and we set 

$$d(\alpha,\beta)=\max \{\ 1/n\ \vert\ \hh_n\ \text{separates}\ \alpha\ \text{and}\ \beta\ \}$$

Now consider the $\hh$ that realizes this maximum.  Any hyperplane separating $g\alpha$ and $g\beta$ is of the form $g\hh$, where $\hh$ separates $\alpha$ and $\beta$.  Note that by part $(2)$ of Lemma \ref{SectorContainment}, the distance $d(g\hh, o)\geq d(g\hh, go)=d(\hh,o)$. It follows
that $d(g\alpha,g\beta)\leq d(\alpha,\beta)$. 

Note also that $GB_i=B$ for all non-empty sets $B_i$ participating in the partition, since $G$ is minimal on $B$ by Theorem 5.8. \end{proof}




\section{Intervals and property $A$}

We now describe the properties of intervals connecting a vertex in the cube complex to a boundary point or two boundary points.  The existence of intervals will prove to be an  indispensable tool in analyzing the properties of the boundary, as we will see below.  The discussion in the present Section pertains to general finite-dimensional locally-finite 
cube complexes, and the group action will not come into play at all.  

Given three ultrafilters $\alpha,\beta,\gamma$, recall that the {\it median} of the three ultrafilters is defined as 

$$med(\alpha,\beta,\gamma)=(\alpha\cap\beta)\cup(\beta\cap\gamma)\cup(\alpha\cap\gamma)$$

It is easy to check that $med(\alpha,\beta,\gamma)$ is itself an ultrafilter. Given two ultrafilters, $\alpha$ and $\beta$, we now define the interval between $\alpha$ and $\beta$ to be 

$$[\alpha,\beta]=\{ \gamma\ \vert\ med(\alpha,\beta,\gamma)=\gamma\}$$
In the case that $\alpha$ and $\beta$ are vertices of $X$, the interval $[\alpha,\beta]$ is simply the collection of vertices that lie on some 1-skeleton geodesic between $\alpha$ and $\beta$ (see \cite{ChN}).

The goal of the present section is to prove the following Folner-type property for such intervals. This will be used later on to show, among other things,  that point stabilizers are amenable. 
Let $B(r)$ denotes the ball of radius $r$ about some base vertex and $\vert\cdot\vert$ the size of the intersection with the 0-skeleton.

\begin{theorem}\label{AmenableInterval}
Let $v,w$ be a vertices in a cube complex $X$ and $\alpha\in\nonterm(X)$ be a nonterminating ultrafilter. Then (for any fixed choice of base vertex as the center of the balls)  

$$\lim_{r\to\infty} \frac{\vert ([v,\alpha]\triangle[w,\alpha]) \cap B(r)\vert }{\vert ([v,\alpha]\cup[w,\alpha])\cap B(r)\vert} = 0\,.$$ 
 
\end{theorem}

\paragraph{Remark} The property stated in the theorem may not hold for $\alpha$  not in $\nonterm(X)$. For example,
consider the standard squaring of the Euclidean plane. There are vertical hyperplanes and horizontal ones. Let $\alpha=(\infty,0)$ be the ultrafilter that contains the right half-space of every vertical hyperplane, the upper half-space of every horizontal hyperplane below the $x$-axis and the lower half-space for every horizontal hyperplane above the $x$-axis.  Let $v=(0,0)$ and $w=(1,0)$. Then it is easy to see that 

$$\lim_{r\to\infty} \frac{\vert ([v,\alpha]\triangle[w,\alpha]) \cap B(r)\vert }{\vert ([v,\alpha]\cup[w,\alpha])\cap B(r)\vert} = 1/2\,.$$

Before beginning the proof, we recall that a theorem of \cite{BC+} tells us that intervals embeds ($\ell_1$-isometrically) in the standard Euclidean cube complex in $\RR^n$. This property is the key to analyzing the properties of intervals, which will occupy us in the next five subsections. The proof of Theorem \ref{AmenableInterval} will conclude in \S 6.6. 
\medskip
\subsection{Basics of Intervals}

We will first need a few basic facts about intervals. 

\begin{lemma}
For any $\alpha, \beta\in\cU(X)$, we have $[\alpha,\beta] = \{\gamma\vert \alpha\cap\beta\subset \gamma\}$
\label{IntervalDescription}
\end{lemma}

\begin{proof}
Observe that 

$$med(\alpha,\beta,\gamma)=\gamma\Leftrightarrow(\alpha\cap\beta)\cup(\beta\cap\gamma)\cup(\alpha\cap\gamma)=\gamma\Leftrightarrow\alpha\cap\beta\subset\gamma$$
\end{proof}

\begin{lemma}
For any $\alpha, \beta, \gamma\in\cU(X)$ and $m=med(\alpha,\beta,\gamma)$, we have that $[\alpha,\gamma]\cap[\beta,\gamma] = [m,\gamma]$.
\label{IntervalIntersection}
\end{lemma}

\begin{proof}
$$\alpha\cap m =\alpha\cap\big(( \alpha\cap\beta)\cup (\alpha\cap\gamma) \cup (\beta\cap\gamma)\big) = (\alpha\cap\beta)\cup(\alpha\cap\gamma)$$
So that by Lemma \ref{IntervalDescription},
$$[\alpha,m]=\{\mu\vert \alpha\cap m \subset \mu\}=\{\mu\vert (\alpha\cap\beta) \cup (\alpha\cap\gamma)\subset\mu\}=[\alpha,\beta]\cap[\alpha,\gamma]$$
\end{proof}

Recall that the carrier $C(\hh)$ of a hyperplane $\hh$ is the union of all cubes intersecting $\hh$. Thus the vertices of $C(\hh)$ may be viewed as the ultrafilters satisfying DCC which contain $\h$ or $\h^*$ as minimal elements. We may then extend the notion of the carrier to all of $\cU(X)$:  an ultrafilter $\alpha$ is in $C(\hh)$ is it contains $\h$ or $\h^*$ as a minimal element. 

We now have the following:

\begin{lemma}
Suppose that $\alpha\triangle\beta<\infty$. Then for any $\gamma\in\cU(X)$, we have 
$$[\alpha,\gamma]-[\beta,\gamma]\subset\bigcup_{\hh \text{\rm \ separating } \alpha \text{\rm \ and } \beta} C(\hh)$$
\end{lemma}

\begin{proof}
Consider $\mu\in [\alpha,\gamma]-[\beta,\gamma]$. Since $\mu \not\in [\beta,\gamma]$, there exists $\h$ such that $\h\in \beta,\gamma$ but $\h^*\in\mu$. Since $\mu\in[\alpha,\gamma]$, we cannot have that $\hh$ separates $\mu$ from both $\alpha$ and $\gamma$. Thus, we have that $h^*\in\alpha$. Notice now that by the same reasoning, for any hyperplane $\hk$ separating $\mu$ and $\hh$, we also have that $\hk$ separates $\beta$ and $\gamma$ from $\mu$ and $\alpha$.

Now there are only finitely many hyperplanes separating $\alpha$ and $\beta$, so choose $\hk$ to be such a hyperplane that is closest to $\mu$. We then have that $\mu\in C(\hk)$, as required. 
\end{proof}

\subsection{Spheres, balls and invariance of basepoint}
From now on in this section, we will let $I$ denote an interval.
We use the following notation for balls and spheres in $I$.

$$B_v(I,r) = \text{ the ball of radius r about v in I}$$ 
$$S_v(I,r)=\text{ the sphere of radius r about v in I}$$
We also use $\abs{\cdot}$ to denote the number of vertices in subset of $I$.

\begin{proposition}
\label{FolnerBalls}
Suppose that $I$ is an interval and $I=[v,\alpha]$, where $v$ is a vertex of $I$. Then

$$\lim_{r\to\infty} \frac{\abs{S_v(I,r)}}{\abs{B_v(I,r)}}=0$$
\end{proposition}

In fact, we shall prove the following more general and sharper statement

\begin{lemma}\label{uniformity}
Let $D$ be a natural number and $\epsilon>0$ a real number. Then there exists a number $R(\epsilon, D)$ such that for any interval of dimension less than or equal to $D$ , we have $\frac{\abs{S_v(I,r)}}{\abs{B_v(I,r)}}<\epsilon$ for all $r>R(D,\epsilon)$.
\end{lemma}

The difference between the proposition and the theorem is that the $R$ chosen for the $\epsilon$ does not depend on the complex, just on the dimension and $\epsilon$. 

\begin{proof}
Since we will always be working with the same vertex $v$ in this lemma, we shall write $S(r)=S_v(I,r)$ and $B(r)=B_v(I,r)$. 

We proceed by induction on the dimension of the complex. For dimension 1, the theorem is clear since $\abs{S(r)}=1$ for all $r$, so  $\frac{\abs{S(r)}}{\abs{B(r)}}$ is either $0$ or $1/r$. 
So we assume that the theorem is true for complexes of dimension less than $D$. In particular we assume the proposition holds for hyperplanes in $I$. 

We assume that $\epsilon$ is given. We make some choices of numbers now that we need for later.  By induction, there exists an $R_1=R(D-1,\epsilon/2)$,  so that for any complex of dimension less than $D$ and any $r>R_1$,  $\frac{\abs{S_r}}{\abs{B_r}}<\epsilon /2$. Let $S_{max}$ denote the maximal number of points in the sphere of radius $R_1$ in a complex of dimension less than $D$, and let $M=R_1\cdot S_{max}$. (The number $S_{max}$ exists because there are only finitely many such possible spheres.)

We choose $R>{\text Max}\{R_1, 2DM/\eps\}$. We claim that $R$ is our required $R(D,\eps)$.

We note that $\RR^D$ can be naturally factorized as a product $\RR^D=\RR^{D-1}\times\RR$ in $D$ different ways. Now we consider our embedding of $I$ into $\RR^D$. We assume that the sphere of radius $R$ is non-empty, for otherwise, the proposition is clearly true. If we take a vertex $w\in I$ such that $d(v,w)=R$, then for one of the projections $\pi: \RR^D\to \RR$ has the property that $d(\pi(v),\pi(w))>R/D>2M/\eps$ (recall that we are using the $\ell_1$ metric, which is the same as the 1-skeleton metric on the vertices of $I$). We factorize $\RR^D=\RR^{D-1}\times\RR$ so that the projection onto the second factor satisfies the above: some point in $S_R$ projects onto a point at distance at least  $2M/\eps$ from $v$. We call the second factor the \emph{vertical} direction and the hyperplanes transverse to them \emph{horizontal} hyperplanes, so that we have at least $2M/\eps$ horizontal hyperplanes intersecting the ball of radius $R$. Let $N$ denote the number of hyperplanes meeting the ball of radius $R$.

We let $L_n$ denote the $n$'th horizontal hyperplane (counting from the bottom). We let $\pi_n:I\to L_n$ denote the projection of $I$ onto $L_n$. Note that $L_n=\pi_n(I)$ is an interval, namely
$L_n=[\pi_n(v),\pi_n(\alpha)]$. Consequently, $B_{r,n}=B_r\cap L_n$ is a ball; we let $d_{n,r}$ denote the radius of this ball. We denote $S_{r,n}=S_r\cap L_n$.  

\paragraph{\bf Observation} Note that for any $n,r$ and $m>n$, we have $d_{m,r}\leq d(n,r)-(m-n)$.

Let $L_k$ be the first hyperplane such that $L_k\cap S_R\not =\emptyset$ and $d_{k,R}\leq R_1$.
By the observation above, we have at most $R_1$ hyperplanes above $L_k$ which intersect the ball of radius $R$. Moreover, for each of these, we have that the radius $d_{n,r}\leq R_1$. Thus
there are at most $M$ points in the spheres $S_{n,R}$ for $n\geq k$. For all $n<k$, we have 
that $S_{n,R}$ is empty or $d(k,R)>R_1$. In either case,  for all $n<k$, we have that $\abs{S_{n,R}}<\frac{\epsilon}{2} \abs{B_{n,R}}$.
We now look at the ratio we are considering and break it up into layers.

$$\frac{\abs{S_R}}{\abs{B_R}}=\frac{ \sum_{i=1}^k \abs{S_{i,R}}+ \sum_{i=k+1}^N \abs{S_{i,R}} }{\sum_{i=1}^N\abs{B_{i,R}}}<
\frac{\frac{\epsilon}{2}\sum_{i=1}^k\abs{B_{i,R}}}{\sum_{i=1}^N\abs{B_{i,R}}}+\frac{M}{N}<\eps/2+\eps/2 =\eps
$$

\end{proof}

\begin{corollary}
\label{BasePointIndependence}
For any two vertices $v$ and $w$ and for any $n\in\ZZ$, we have 

$$\lim_{r\to\infty} \frac{B_v(I,r+n)}{B_w(I,r)}\to 1$$
\end{corollary}

\begin{proof}

First,  we need to see that additive constants to not affect the limit. This follows immediately from Proposition \ref{FolnerBalls}, since it tells us that

$$\lim_{r\to\infty}\frac{\abs{B_v(I,r)}}{\abs{B_v(I,r+1)}} =1.$$

Secondly, we show independence of basepoint. Let $d=d(v,w)$. Then for any $r>d$, we have 

$$B_w(I,r-d)\subset B_v(I, r)\subset B_w(I, r+d)$$
We thus have 

$$\frac{\vert B_w(I,r-d)\vert}{\vert B_w(I,r)}\leq \frac{\vert B_v(I,r)\vert}{\vert B_w(I,r)}\leq \frac{\vert B_w(I, r+d)\vert}{\vert B_w(I,r)\vert}$$

By the lemma, we have the left and right hand sides approach 1, as $r\to\infty$, so this gives 

$$\lim_{r\to\infty} \frac{\abs{B_v(I,r)}}{\abs{B_w(I,r)}}\to 1$$

The corollary follows.
\end{proof}

\subsection{Volume and collapsing} 
We will need to understand something about the size of the portion of the interval taken up by its intersection with a hyperplane.  Given an interval $I$ and $\hh$ a hyperplane in $I$, we say that $\hh$ is {\it meager} if  (for any fixed choice of base vertex as the center of the balls)  
$$\lim_{r\to\infty} \frac{\vert C(\hh)\cap B_v(I,r)\vert }{\vert B_v(I,r)\vert } = 0\,.$$

\begin{remark}
Let us note that by Corollary \ref{BasePointIndependence} the property of a hyperplane being meager is independent of the base point. 
\end{remark}

We now recall the construction called {\it collapsing}, discussed in \S 2.3. 
Since $C(\hh)\cong \hh\times I$, we may form a quotient of $I$ by collapsing the interval direction of $C(\hh)$ to a point (i.e. via projection onto the $\hh$ factor). The resulting complex $\overline I$ is  called the complex obtained by {\it eliminating $\hh$}. 
The quotient map $I\twoheadrightarrow\overline I$ is called the {\it collapsing} map for $\hh$.  One may also eliminate a finite number of hyperplanes. It is easy to check that such a collapsing map is well-defined in that it does not depend on the order in which the hyperplanes are collapsed. 

We will need to understand what collapsing does to growth of balls. 

\begin{lemma}
Let $\rho: I\onto  J$  be a collapsing for a single hyperplane and let $v$ be a vertex in $I$. Then for all $r>0$, we have

$$ \frac{\vert B_{\rho(I)}(J,r)\vert }{\vert B_v(I,r)\vert } \geq 1/2.$$ 
\end{lemma}

\begin{proof}
First we observe that since the distance between vertices simply counts the number of hyperplanes which separate them, we have that $\rho$ is distance non-increasing, so that 

$$ \rho(B_v(I, r) )\subset B_{\rho(v)}(J, r)$$  
Now we note that the collapsing map is at most 2-to-1, so that $$2\vert B_{\rho(v)}(J,r)\vert \geq \vert B_v(I,r)\vert$$ 

It follows that 

$$\frac{\vert B_{\rho(v)}(J,r)\vert }{\vert B_v(I,r)\vert }\geq 1/2$$
\end{proof}

\begin{lemma}
Let $\rho: I \onto J$ denote the collapsing map eliminating a meager hyperplane $\hh$.
Then 
$$\lim_{r\to\infty} \frac{\vert B_{\rho(v)}(J,r)\vert }{\vert B_v(I,r)\vert} = 1$$ 
\end{lemma}

\begin{proof}

As in the previous lemma, we have $\rho(B_v(I,r))\subset B_{\rho(v)}(J,r)$. Also, note that $\rho$ is 1-1 on the complement of $C(\hh)$. Thus, we obtain: 

$$\frac{\vert B_{\rho(v)}(J,r) \vert}{\vert B_v(I,r)\vert} \geq \frac{\vert \rho(B_v(I,r))\vert}{\vert B_v(I,r)\vert}   \geq \frac{\vert \rho\big(B_v(I,r)-C(\hh)\big)\vert}{\vert B_v(I,r)\vert} = \frac{\vert B_v(I,r)-C(\hh)\vert}{\vert B_v(I,r)\vert} $$

By the definition of meager, we have that $$\lim_{r\to\infty} \frac{\vert B_v(I,r)-C(\hh)\vert}{\vert B_v(I,r)\vert} = 1$$ The lemma follows. 
\end{proof}

\begin{corollary}\label{FiniteCollapse}
Let $I=[v,\alpha]$ be an interval with no more than $n$ non-meager hyperplanes. Then for any collapsing map of finitely many hyperplanes $\rho:I\onto J$, we have 
$$\liminf_{r\to\infty} \frac{\abs{B_{\rho(v)}(J,r)}}{\abs{B_v(I,r)}}\geq 1/2^n$$

\end{corollary}

\subsection{Projections of hyperplanes onto one another}
\label{subembeddings}

As above, we let $I=[v,\alpha]$  is an interval embedded in $R^D$, with $v$ the origin and $\alpha$ an nonterminating ultrafilter. As before, we write $R^D=R^{D-1}\times R$, with the last factor being the vertical direction and the hyperplanes transverse to this direction called horizontal hyperplanes. As before, the hyperplanes are ordered from their distance from the origin and 
are denoted $L_n$. 

As discussed previousely, 
there exists a projection of $I$ to $L_n$ . Restricting, this we get a projection map $\pi_n: L_1\to L_n$. 
We want to get a better handle on this map.

As usual, let $\hat\cH(L_n)$ denote the collection of hyperplanes of the complex $L_n$: these are simply the hyperplanes of $I$ which cross $L_n$. 
Note that any hyperplane $\hh$ which crosses $L_1$ and does not cross $L_n$ must separate $v$ from $L_n$. This is because both $\hh$ and $L_n$ must separate $v$ from $\alpha$. Thus, there are finitely many hyperplanes which cross $L_1$ and do not cross $L_n$. We let $\rho_n: L_1\onto Y_n$ denote the resulting collapsing map. 

Now for any hyperplane $\hh$ in $I$ crossing $L_n$ which does not cross $L_1$, we let $\h$ denote the halfspace containing $L_1$. We let 

$$C_n=\bigcap_{\hh\cap L_1=\emptyset} \h\cap L_n$$

The subcomplex $C_n$ is simply the image $\pi_n(L_1)$.
Note that $C_n$ is itself a convex subcomplex of $L_n$ whose hyperplanes consist of those hyperplanes in $I$ which intersect $L_1$.  Thus $Y_n$ and $C_n$ have the same halfspace system, namely the one coming from all hyperplanes which cross both $L_1$ and $L_n$. We thus have a natural isomorphism $i_n:Y_n\to C_n$.  We see then that $\pi_n=i_n\circ \rho_n$. 


We summarize the above discussion in the following lemma. 

\begin{lemma}
For each $L_n$, the projection map $L_1\to L_n$ factors through a collapsing map $\rho: L_1\onto Y_n$ and an embedding $Y_n\to L_n$. 

\label{SubembeddingHyperplanes} 
\end{lemma}

\subsection{Non-meager hyperplanes are inessential}

\begin{theorem}
Suppose that $I$ is an interval. Then 
\begin{enumerate}
\item $I$ contains only finitely many non-meager hyperplanes.
\label{FiniteNonMeager}
\item For every non-meager hyperplane $\hh$ in $I$, there exists some $R$ such that the $R$-neighborhood of $\hh$ is $I$
\label{NonMeagerIsEssential}
\end{enumerate}

\end{theorem}

\begin{proof}
We prove (1) by induction on $D$, the dimension of $I$. For $D=0$, the statement is trivial.  

Suppose that $I$ has infinitely many non-meager hyperplanes. Since $I$ is finite dimensional we may choose these hyperplanes $L_1,L_2, \ldots$ that they are disjoint. We further choose the $L_i$'s so that their carriers are disjoint. As in Section \ref{subembeddings}, we view $I$ as being embedded in $R^D$ and we may pass to a subsequence of the $\{L_n\}$ so that they are all transverse to some vertical direction of a factorization $\RR^D=\RR^{D-1}\times \RR$. We also order $\{L_n\}$ by distance to the base vertex $v$ and let $v_n$ denote the projection of $v$ onto $v_n$, for each $n$. We let $d_n=d(v,v_n)$

Since $L_1$ is not meager, there exists some number $\eps>0$ and a sequence $\{r_k\}$ such that 

$$\frac{\abs{B_v(I,r_k)\cap C(L_1)}}{\abs{B_v(I,r_k)}}>\eps$$

Since we are using the $1$-skeleton metric, we have that for all $n$, 
$$B_v(I,r)\cap L_1 = B_{v_n}(L_n, r-d_n)$$
Furthermore, we know that the natural map $C(L_n)\to L_n$ is 2-1 and thus between 2-1 and 1-1 in the intersection with $B(I,r_k)$, we obtain 
\begin{equation}
\label{eq1}\frac{\abs{B_{v_1}(L_1, r_k-d_1)}}{\abs{B_v(I,r_k)}}=\frac{\abs{B_v(I,r_k)\cap L_1}}{\abs{B_v(I,r_k)}}\geq\frac{\frac{1}{2}\abs{B_v(I,r_k)\cap C(L_1)}}{\abs{B_v(I,r_k)}} >\eps/2
\end{equation}

By induction we know that $L_1$ has only finitely many non-meager hyperplanes. So let $N$ denote the number of non-meager hyperplanes in $L_1$. By Lemma \ref{SubembeddingHyperplanes}, we have that the projection $L_1\to L_n$ factors through a collapse of finitely many hyperplanes and an embedding. We may thus apply Lemma \ref{FiniteCollapse} to obtain:

$$\frac{\abs{B_{v_n}(L_n, r_k-d_1)}}{\abs{B_{v_1}(L_1, r_k-d_1)}} \geq 1/{2^N}$$

By Lemma \ref{BasePointIndependence}, for $r_k$ sufficiently large, we obtain 

\begin{equation}\label{eq2}
\frac{\abs{B_{v_n}(L_n, r_k-d_k)}}{\abs{B_{v_n}(L_n, r_k-d_1)}}>1/2
\end{equation}
So that for $r_k$ sufficiently large, we obtain

$$\frac{\abs{B_{v_n}(L_n, r_k-d_k)}}{\abs{B_{v_1}(L_1, r_k-d_1)}}>1/{2^{N+1}}$$

We choose $M>2^{N+1}/\eps$ and choose $r_k$ sufficiently large so that for all $1\leq n\leq M$, equation \ref{eq2} holds.
Putting this together with inequality \ref{eq1}, we obtain for all $1\leq n\leq M$, 

$$\frac{\abs{B_v(I,r_k)\cap C(L_n)}}{\abs{B_v(I,r)}} > \frac{\abs{B_{v_n}(L_n, r_k-d_k)}}{\abs{B_v(I,r)}}>\eps/{2^{N+1}}$$

Summing over the first $M$ levels, we obtain 

$$\sum_{n=1}^M \frac{\abs{B_v(I,r_k)\cap C(L_n)}}{\abs{B_v(I,r_k)}}>M\eps/2^{N+1}>1$$

a contradiction.

To prove (2), we proceed in the same manner. Suppose there was a non-meager hyperplane in $I$ which did not have a neighborhood containing $I$. Then, calling that hyperplane $L_1$, we would then have an infinite sequence of disjoint hyperplanes $L_1, L_2, \ldots$ in $I$. Now the same projection arguments as above show that each $L_n$ is non-meager, contradicting part (1).
\end{proof}
\bigskip

\begin{corollary}
If $E=[p,\alpha]$, where $p\in I$ and $\alpha$ is a nonterminating ultrafilter, then all the hyperplanes in $E$ are meager.
\label{AllMeager}
\end{corollary}

\begin{proof}
By the definition of nonterminating, for every hyperplane $\hh$, there exists another hyperplanes $\hh'$ separating $\hh$ from $\alpha$. Therefore no neighborhood of $\hh$ contains all of $I$.
\end{proof}

%


%

%







\subsection{Proof of Theorem \ref{AmenableInterval}}

Finally, we are now ready to prove the main theorem of this section. 
\begin{proof}[Proof of Theorem \ref{AmenableInterval}] Let $m = med(v,w,\alpha)$. Then we have that 

$$[v,\alpha]\cap[w,\alpha]=[m,\alpha]$$

Also,  $[v,\alpha]\triangle [w,\alpha] = [v,\alpha]\cup[w,\alpha] - [m,\alpha]$, so that in order to show our result, it suffices to show that 

$$\lim_{r\to\infty} \frac{ \vert ( [v,\alpha]-[m,\alpha]) \cap B(r)  \vert  }{\vert [m,\alpha]\cap B(r)  \vert   } = 0$$

Let $\hh_1,\ldots, \hh_k$ denote the hyperplanes separating $v$ and $m$. Then $$[v,\alpha]-[m,\alpha]\subset \bigcup_{i=1}^k C(\hh_i)$$
Thus, we conclude that independently of the base point defining $B(r)$ 

$$\lim_{r\to\infty}  \frac{ \vert ( [v,\alpha]-[m,\alpha]) \cap B(r)  \vert  }{\vert [m,\alpha]\cap B(r)  \vert   } \leq \lim_{r\to\infty} \frac{\vert ( \bigcup_{i=q}^k C(\hh_i ) \cap B(r) \vert }{  \vert [m, \alpha] \cap B(r)   \vert  } =0,$$
 by Corollary \ref{AllMeager}.
\end{proof}

\subsection{Property $A$ for groups acting on finite-dimensional cube complexes}

It was shown in \cite{BC+} that a finite-dimensional cube complex satisfies Yu's Property $A$. 
We give here an alternative proof in the case the cube complex is in addition locally finite. Indeed,  let us first  assume 
that the cube complex is irreducible non-Euclidean. In that case, $\nonterm(X) $ is non-empty, namely there is a non-terminating ultrafilter $\alpha$. For any vertex $v\in X$, consider the sets $B_n(v)\cap [v,\alpha]=A_{v,n}$, namely the intersection of the interval from $v$ to $\alpha$ with a ball of radius $n$ and center $v$ in $X$. Then by Theorem \ref{AmenableInterval}  for any $\vare > 0$ the sets satisfy 
$$\frac{\abs{A_{v,n} \Delta A_{w,m}}}{\abs{A_{v,n}}}< \vare, \text{ provided } m\ge n \ge n(\vare)$$ 
and of course $\text{diam } A_{v,n}\le 2n$. This sequence of sets constitutes a direct generalization of the sequence usually used to show that a tree has property $A$, namely the sequence consisting of intervals of length $n$ on the unique geodesic from a vertex $v$ to a boundary point $\alpha$. The existence of such a sequence implies property $A$ for the cube complex, by definition. Now using the product decomposition theorem of \cite{CS}  the result follows for every finite-dimensional locally finite \cat cube complex, recalling also 
that Euclidean complexes have Property $A$,  since they admit a Folner sequence.

We now turn to another ingredient that plays an important role in the proof of Theorem \ref{B(X) UMSP}.

\section{Amenability of the boundary action}

In the present section we will discuss the amenability properties of the  action of discrete groups on the boundary of a cube complex.  The key to our analysis is the existence of intervals connecting vertices in the complex to boundary points, with balls in the intervals satisfying the Folner property established in  Theorem \ref{AmenableInterval} and Corollary 
\ref{BasePointIndependence}. Later on, in our discussion of the Poisson boundary, we will also use the fact that intervals embed in $\RR^n$ and hence have polynomial growth. 

Let us first recall the following definitions of amenability of an action. 
Let $P(G)$ denote the space 
of probability measures on the countable group $G$ taken with the total variation norm ($\ell^1(G)$-metric). $G$ acts on $P(G)$ by translations, denoted $\mu\mapsto g\mu$, and the action is continuous. We denote by $P_c(G)$ the subspace of finitely supported measures.

\begin{definition}\label{TopAmen}{\bf Topological, Borel and universal amenability.}

\begin{enumerate}
\item 
The $G$ action on a locally compact metric space $B$ is called topologically  amenable, if we can find a sequence 
of continuous  functions $\beta_n:B\to P_c(G)$, such that 
$$
\lim_{n\to \infty}\norm{\beta_n(gb)-g\beta_n(b)}=0\,\,
$$
for every $b\in B$, and the convergence is uniform on compacts subsets in $B$. 
\item The $G$ action on a standard Borel space $B$ is called Borel amenable if  we can find a sequence of Borel measurable functions $\beta_n:B\to P_c(G)$ which satisfies 
$\lim_{n\to \infty}\norm{\beta_n(gb)-g\beta_n(b)}=0$ 
for every $b\in B$. 
\item The $G$ action on a standard Borel space $B$ is called universally (or measure-wise) amenable if for every $G$-quasi-invariant probability measure 
$\eta$ on $B$, we can find a sequence of measurable functions $\beta_n$, which are defined and satisfy $\lim_{n\to \infty}\norm{\beta_n(gb)-g\beta_n(b)}=0$ for $\eta$-almost all points $b\in B$.   
\end{enumerate} 
\end{definition}

The space $B$ we will consider to begin with is $\nonterm(X)$, which by  Proposition \ref{DenseGDelta} is a dense  $G_\delta$ of the boundary $B(X)$.

\begin{theorem}\label{universal}
Let $X$ be an irreducible, non-Euclidean \cat cube complex.  Then for any discrete subgroup $G \subset \Aut(X)$ acting properly, the action on $\nonterm (X)$ is Borel amenable, and hence also a universally (or measure-wise) amenable action. 
\end{theorem}
\begin{proof}
We will first construct a sequence of Borel maps $\tilde{\beta}_n : \nonterm (X)\to P(V)$, where $V=V(X)$ is the set of vertices of the complex $X$, taking values in finitely-supported probabilty measures on the vertex set $V(X)$. 

To define $\tilde{\beta}_n (b)$ when $b$ is a nonterminating ultrafilter, fix a reference vertex $o\in V(X)$. Given any point $b\in \nonterm (X)$, draw the interval from 
$o$ to $b$, denoted $[o,b)$. Define $\tilde{\beta}_n : \nonterm (X)\to P(V)$, where $\tilde{\beta}_n(b)$ is the probability measure uniformly distributed on the finite set of vertices obtained as the intersection of a ball of radius $n$ with center $o$ and the interval $[o,b)$  from $o$ to $b$. 

Now note that $\tilde{\beta}_n(gb)$ is the measure on the complex uniformly distributed 
on $B(n,o)\cap[o,gb)$. On the other hand, $g\tilde{\beta}_n(b)$ is the measure uniformly distributed on 
$B(n,go)\cap [go,gb)$. 

In view of Theorem \ref{AmenableInterval}, given $\vare > 0$, there exists an $n$ sufficiently large such that outside the ball $B(n,o)$, the symmetric difference between the intervals 
$[o,gb)$ and $[go,gb)$ has size bounded by $\vare$ times the sum of their sizes.

It follows the difference between the measures namely 
$\norm{\tilde{\beta}_n(gb)-g\tilde{\beta}_n(b)}$ (in $\ell^1 (V(X))$-norm), does indeed converge to zero, for any  given $g$ in $G$. 

Let us also note that the convergence is in fact uniform on $\nonterm(X)$ by Lemma \ref{uniformity}, 
namely for a fixed $\vare > 0$, $n$ can be chosen independent of $b\in \nonterm(X)$. 

We now pass to a discrete subgroup $G\subset \Aut(X)$ acting properly by restricting each measure
$\tilde{\beta}_n$ to the various $G$-orbits in the complex, renormalizing and viewing the resulting measure 
as an element of $P(G)$. More explicitly, let $T$ be a transversal to the $G$-orbits in $V(X)$, and for $v\in T$ let $G_v$ be the stability group of $v$ in $G$, which is finite since the action of $G$ is proper. 
Let us view the probability measure $\tilde{\beta}_n(b)$ as function on the vertices in $V(X)$, and define   
$$\beta_n(b,g)=\sum_{v\in T} \frac{\tilde{\beta}_n(b,gv)}{\abs{G_v}}\,.$$
Then the $\ell^1(G)$-norm of $\beta_n(b,\cdot)$ is 
$$\norm{\beta_n(b,\cdot)}_{\ell^1(G)} =\sum_{g\in G} \beta_n(g,b)=\sum_{g\in G,v\in T} \frac{\tilde{\beta}_n(b,gv)}{\abs{G_v}}=\sum_{v\in T}\sum_{w\in G\cdot v} \tilde{\beta}_n(b,w)\sum_{g\in G : gv=w}\frac{1}{\abs{G_v}}$$
$$
=\sum_{v\in T}\sum_{w\in G\cdot v} \tilde{\beta}_n(b,w)=\norm{\tilde{\beta}_n(b,\cdot)}_{\ell^1(V)}\,.
$$
Clearly a similar computation shows that $\norm{g\beta_n(b)-\beta_n(gb)}_{\ell^1(G}\to 0$ and thus 
$\nonterm(X)$ is a Borel-amenable action of $G$. 
\end{proof}

%
A simple consequence of Theorem  \ref{universal} is the following 

\begin{corollary}\label{stability}
The stability group of a point in $\nonterm (X)$, namely of a nonterminating ultrafilter in a (finite dimensional) \cat cube complex  is an amenable group.
\end{corollary}
\begin{proof}
Let $g\in S=St_G(b)$ be a group element which stabilized $b\in \nonterm(X)$. Then 
the sequence of finitely supported probability measures $\beta_n(b)\in P_c(G)$ is clearly asymptotically invariant under left translation by $g$. The sequence defines an mean on $\ell^\infty(G)$ in the usual way, and this mean is invariant under the subgroup $S$. Extending a bounded function on $S$ to a bounded function on $G$ by transfering it to the other cosets in the obvious way, we get an $S$-invariant mean on $\ell^\infty(S)$, so $S$ is amenable. 
\end{proof}

We note that Corollary \ref{stability} is in fact a consequence of more general results, considered in various formulations in \cite{CN},  \cite{Ca07} and \cite{BC+}. 
Of those, we quote the result of \cite{BC+} most pertinent to us. 
 \begin{theorem}
Let $X$ be a (finite dimensional) \cat cube complex. 
 The stability group of every ultrafilter, namely every point in the Roller boundary and thus in particular every point in  $B(X)$,  is an amenable group.
\end{theorem}

\begin{remark}\
\begin{enumerate}

\item In \cite{Ca07} a structure theorem is proved for amenable closed subgroups of a group that acts properly, co-compactly and discontinuously on Hadamard spaces, and in particular on finite-dimensional locally finite  \cat cube complexes. It is shown that an amenable  group virtually admits a homomorphism into $\RR^d$, with the kernel being a topologically locally finite group. 

\item We do not know whether  the action of $G$ on $B(X)$ is topologically amenable. The remark following Theorem \ref{AmenableInterval} casts some doubt whether this can be true, but is not conclusive. 

\end{enumerate}
\end{remark}

\section{Maximality of the boundary}
 
 In the present section we will show that the boundary $B(X)$ with its unique stationary measure $\nu$ 
 is a compact metric model of the Poisson boundary $\cB(G,\mu)$. We will use in our analysis an important criterion for boundary maximality developed by V. Kaimanovich, called the strip criterion. This criterion is applicable in the context of cube complexes, since intervals provide us with a natural notion of strips in the complex. In fact, any two distinct ultrafilters $\alpha\neq \beta$ determine a unique interval of ultrafilters  between them, defined by 

 $$[\alpha,\beta]=\{\gamma\vert m(\alpha,\beta,\gamma)=\gamma\}\,.$$
 By Lemma \ref{IntervalDescription}, we have that 
 
 \begin{equation}\label{IntervalAsIntersection}
[\alpha,\beta]=\bigcap_{\h\in\alpha\cap\beta} U_\h
\end{equation}
The map $B(X)\times B(X) \to 2^{\cU(X)}$ defined by   $(\alpha,\beta)\mapsto [\alpha,\beta]$ is of course $\Aut(X)$-equivariant.   
 
 \subsection{Generic pairs of ultrafilters}
 
  It may happen, however, that the interval consists of non-principal ultrafilters, i.e.  the strip between two non-principal ultrafilters $ \alpha$ and $\beta$ may lie itself "at infinity". This arises even in simple and natural examples, such as $\ZZ^2$ and $T_3\times T_3$.

Thus not all pairs of boundary points are alike, and we must find the right notion a ``generic pair". A pair $(\alpha,\beta)\in B(X)\times B(X)$ is called \emph{generic} if $\cS(\alpha,\beta)\equiv[\alpha\cap\beta]\cap X\not=\emptyset$. The set $\cS(\alpha,\beta)$ is called the \emph{strip} between $\alpha$ and $\beta$. We let $\cG\subset B(X)\times B(X)$ denote the collection of generic pairs. We then have the following.

\begin{proposition}\label{GenericOpen}
Let $X$ be a strictly non-Euclidean, cocompact \cat cube complex. 
Then the set $\cG$ is a non-empty open invariant subset in $B(X)\times B(X)$.
\end{proposition}

\begin{proof} 
First, we need to show that $\cG$ is non-empty. First observe that it suffices to show this for the case that $X$ is irreducible, because the generic pairs in the products appear as products of generic pairs in the factors. 

So now assume that $X$ is irreducible and suppose that $\cG=\emptyset$. 
Let $\alpha, \beta$ be a pair of distinct nonterminating ultrafilters. By assumption we have that $S(\alpha,\beta)\cap X=\emptyset$ for every such pair. 

Since $\alpha$ and $\beta$ are nonterminating ultrafilters, there exist infinitely many hyperplanes separating them. In particular there exists a pair of disjoint hyperplanes separating them.  We thus see that given any two such $\alpha$ and $\beta$, we have that

\begin{enumerate}
\item there exists a collection of intersecting hyperplanes $\hh_1,\ldots,\hh_n$ such that
$\alpha\in\h_i$ and $\beta\in\h_i^*$,
\item there exists $\hk$ such that $\hk\in\h_i^*$ 
\item $\alpha\in\k$ and $\beta\in\k^*$.
\end{enumerate}

The above remark tells us that $(1)-(3)$ holds for $n=1$. The number $n$ is bounded by the dimension of $X$, so that we may choose $\alpha$ and $\beta$ such that $n$ is maximal. 

Note that by \ref{IntervalAsIntersection}, we have that

$$S(\alpha,\beta)=\bigcap_{\h\in\alpha\cap\beta} \h$$

Note that if $\alpha\cap\beta=\emptyset$, then all hyperplanes of $X$ separate $\alpha$ and $\beta$. In this case clearly every vertex of $X$ is in $S(\alpha,\beta)$ and we are done. 

So suppose that the above intersection is indeed an intersection of a non-empty collection of halfspaces. By Corollary \ref{NonTrivialIntersection} if $\cK=\{\h \vert \h\in\alpha\cap\beta\}$ satisfies DCC, then we would have that $S(\alpha,\beta)\not=\emptyset$. So by our assumption, we have that $\cK$ does not satisfy DCC. Let $\m_1\supset\m_2,\ldots$ denote a nonterminating sequence of halfspaces in $\cK$. 

We claim that there exists some $l$ such that $\hm_l$ intersects both $\hk$ and $\hh_i$ for all $i$. 

Since each $\hm_n\in\alpha\cap\beta$ and $\hk$ separates $\alpha$ and $\beta$, it follows that if $\hk\cap\hm_l=\emptyset$ then $\hk\subset\m_l$. But since there is a finite distance between $\hm_1$ and $\hk$, we must have that for some $l$, $\hm_l\cap\hk\not=\emptyset$. Similarly, we can choose $l$ large enough so that both $\hm_l\cap\hk\not=\emptyset$ and $\hm_l\cap\hh_i\not=\emptyset$ for all $i$. 

We can further choose $l$, so that the half-hyperplane $\hk\cap\m^*_l$ is deep. We then consider the sector $\m_l^*\cap\k^*$. By Lemma \ref{Corner}, there exists some hyperplane $\hk_1\subset\m_l^*\cap\k^*$. We let $\k_1$ denote the halfspace of $\hk_1$ contained in $\alpha$. By Remark \ref{ManyIDU}, there exists a nonterminating ultrafilter $\gamma\in\k_1^*$. 

Now observe that for the pair $\alpha, \gamma$, the collection of hyperplanes $\hh_1,\ldots,\hh_n,\hm_l$ and the hyperplane $\hk_1$ all satisfy properties $(1)-(3)$ above. But this contradicts the maximality of $n$ as chosen. 

This completes the proof that $\cG\not=\emptyset$. To prove that it is open we will use that the median map $m:\cU(X)\times\cU(X)\times\cU(X)\to \cU(X)$ is continuous in the Tychonoff topology. 
Suppose that $(\alpha, \beta)$ is a generic pair. Then there exists $v\in X^{(0)}$ such that 
$med(\alpha,\beta, v)=v$. Since vertex singletons in $X^{(0)}$ are open, we may find open sets 
$U$ of $\alpha$ and $V$ of $\beta$ such that for any $\alpha^\prime\in U$, and $\beta^\prime\in V$, we have $med(\alpha^\prime,\beta^\prime, v)=v$. Thus $S(\alpha^\prime,\beta^\prime)\not=\emptyset$, as required.  Finally, $\cG$ is clearly an invariant set under the product action of $\Aut(X)$ on $B(X)\times B(X)$. 
\end{proof}


\begin{remark}
Since $\cG$ is $\Aut(X)$-invariant, any $\sigma$-finite measure $\eta$ on $B(X)\times B(X)$ which is quasi-invariant and ergodic under a discrete subgroup $G$ and charges every open set, must assign $\cG$ full measure. 
\end{remark}

\subsection{Boundary maximality via the strip criterion}

We will now use our construction of strips in the complex, for every generic pair of points in $B(X)$, in order to show that $B(X)$ is a maximal boundary, namely realizes the Poisson boundary. 

First, recall that a probability measure $\mu$ on (a countable group) $G$ is called of  finite logarithmic moment if there exists 
a distance function $\abs{g}$ which is quasi isometric to a word metric, and satisfies 
$\sum_{g\in G}\log\abs{g} \mu(g) < \infty$.  This definition is independent of the word-metric chosen. 

Recall also that the Avez entropy of a probability measure $\mu$ on (a countable group) $G$ is defined by 
$$H(\mu)=\lim_{n\to \infty}\frac{-1}{n}\sum_{g\in G} \mu^{\ast n}(g)\log \mu^{\ast n}(g)\,\,.$$
This quantity is equal to the $\mu$-entropy of the stationary measure $\nu$ on the Poisson boundary, see below. 

The measure $\hat{\mu}$ is defined by $\hat{\mu}(g)=\mu(g^{-1})$,  and it has finite logarithmic moment and finite entropy if $\mu$ does.

Now recall the following criterion for boundary maximality, due to V. Kaimanovich.

\begin{theorem}\label{kaimanovich}{\bf Strip criterion.}
Let $\mu$ be a probability measure of finite first logarithmic moment and finite entropy on a group $G$. Assume that 
$(B,\nu)$  is a $(G,\mu)$-boundary, and $(\hat{B},\hat{\nu})$ is a $(G,\hat{\mu})$-boundary. Assume there exists a measurable map defined on $(B \times \hat{B},\nu\times \hat{\nu})$,   denoted $(\alpha,\beta)\mapsto \cS(\alpha,\beta)\subset G$  (viewed as assigning strips to pairs of boundary points). If for $\nu\times \hat{\nu}$-almost all pairs $(\alpha,\beta)$ the strip $\cS(\alpha,\beta)$ has polynomial growth (w.r.t. the distance function above), then both 
$(B,\nu)$ and $(\hat{B},\hat{\nu})$ are maximal boundaries, namely they realize the Poisson boundaries 
of $(G,\mu)$ and $(G,\hat{\mu})$.   
\end{theorem}

We will use this criterion to prove :
\begin{theorem}\label{Poisson}
Let $X$ be an irreducible non-Euclidean \cat cube complex. Let $G$ be a discrete 
subgroup of $\Aut(X)$ acting properly and co-compactly on the complex. Let $\mu$ be a probability measure on $G$ of finite logarithmic moment and finite entropy, whose support generates $G$ as a semigroup. Denote the unique $\mu$-stationary measure on $B(X)$ by $\nu$.  Then $(B(X),\nu)$ is a compact metric model of the Poisson boundary of $(G,\mu)$. 
\end{theorem}

\begin{proof}
  
 Any stationary measure is $G$-quasi-invariant, and hence its support is a closed non-empty $G$-invariant set. In Theorem \ref{B(X) limit set} we have shown that the $G$-action on $B(X)$ is minimal, and hence the support of a quasi-invariant measure coincides with $B(X)$. In particular, the measure of every non-empty open set is strictly positive. 

Consider now the $\Aut(X)$-equivariant map $(\alpha,\beta)\mapsto \cS(\alpha,\beta)$ from 
$B(X)\times B(X)$ to strips in $X$. We first claim that for a set of pairs $(\alpha,\beta)\in B(X)\times B(X)$ of $\nu\times \hat{\nu}$-measure $1$, the strips are actually subsets of the complex, namely  that the set of $\cG$ of generic pairs has full $\nu\times\hat{\nu}$-measure. 

The action of the product of the Poisson boundaries associated with $\mu$ and $\hat{\mu}$  is ergodic, in general (see \cite{k}), 
and thus so is the action 
on any of its $G$-factor spaces. By Theorem  \ref{B(X) UMSP}  the space $B(X)$ is a $(G,\mu)$-boundary and thus $(B(X),\nu)\times (B(X),\hat{\nu})$ is a factor of the product of  
the Poisson boundaries, so that $\nu\times\hat{\nu}$ on $B(X)\times B(X)$ is ergodic. 

The set $\cG$  of generic pairs is invariant under the $G$-action on $B(X)\times B(X)$, and clearly has positive $\mu\times\hat{\mu}$-measure, since it is non-empty and open, and the product measure charges every non-empty open set. By ergodicity, $\cG$ has full  $\nu\times \hat{\nu}$-measure. 

We now pass to strips in the group $G$ itself rather than in the complex. To that end, note that by co-compactness of $G$, it has only finitely many orbits of vertices in $X$, and let us call the different orbits "types" denoted $t_1,\dots ,t_r$, and choose a vertex in each orbit $v_1,\dots,v_r$. For each generic pair $(\alpha,\beta)$, the strip $\cS(\alpha,\beta)\subset X$ decomposes to a disjoint union of vertices belonging to these types.  Clearly there exist at least one type (say $t_1$) such that for a positive measure subset of generic pairs, the associated strip contains  vertices of  type $t_1$. But the latter set of generic pairs is clearly $G$-invariant, and so necessarily has full measure. We now define the map  $(\alpha,\beta)\mapsto \cS^\prime(\alpha,\beta)\subset G$ to strips in the group as follows. 
 $S^\prime(\alpha,\beta)$ is defined to be the union of all the cosets 
$gSt_G(v_1)$, as $g$ ranges over all group elements with the property that $gv_1$ is in $\cS(\alpha,\beta)$, namely $gv_1$ is a vertex of type $t_1$ in that strip. The map is clearly equivariant, and the stability group $St_G(v_1)$ is finite. Now $\cS(\alpha,\beta)\subset X$ has polynomial growth since it is contained in an interval and thus embeds in $\RR^d$ \cite{BC+}. The polynomial growth is with respect to the $\ell^1$-metric on vertices in $X$, and hence $\cS^\prime(\alpha,\beta)$ has polynomial growth with respect to a distance function quasi-isometric to a word metric on $G$. The desired result now follows from Theorem \ref{kaimanovich}. 
%
\end{proof}

\section{Entropy and the Poisson boundary}
Theorem \ref{Poisson} establishes, in particular, that for measures $\mu$ on $G$ with finite logarithmic moment, $B(X)$ gives rise to a compact metric uniquely-stationary model of the  Poison boundary. In the present section we would like to demonstrate that  
the existence of such a model for the Poisson boundary is a significant fact, which has important consequence for the boundary theory of a countable group $G$. 

First let us  recall the definition of $\mu$-entropy of a standard Borel $(G,\mu)$-space (see \cite{nz1} for a detailed discussion) 
\begin{definition}\label{mu-entropy} 
The $\mu$-entropy of a $(G,\mu)$-space $(B,\nu)$ is defined for a countable group $G$ by 
$$h_\mu(B,\nu)=\sum_{g\in G }\mu(g)\int_B-\log \frac{dg^{-1}\nu}{d\nu}(b)d\nu(b)$$
\end{definition}
As noted in \S 8, the Avez entropy $H(\mu)$ coincides with the $\mu$-entropy of the Poisson boundary. Furthermore, this value constitutes the largest value that the $\mu$-entropy can assume, ranging over all $(G,\mu)$-spaces $(B,\nu)$. 
Recall that the Poisson boundary is the  unique (up to $\nu$-null sets) maximal standard Borel $(G,\mu)$-space which is a $\mu$-boundary. Here maximality means that any other $(G,\mu)$-space is a factor of the Poisson boundary, with the factor map  uniquely determined, up to $\nu$-null sets.  
Recall also that the action on the Poisson boundary is amenable in the sense of Zimmer (see \cite{Z2} for  a detailed discussion).

\begin{theorem}\label{unique stationarity}{\bf Characterization of the Poisson boundary.}

Let $G$ be a countable group, $\mu$ a probability measure whose support generates $G$ as a semigroup, and assume that there exists a compact metric $G$-space $B$, which admits a unique $\mu$-stationary measure $\nu$, such that $(B,\nu)$ realizes the Poisson boundary of $(G,\mu)$. Then every amenable $(G,\mu)$-space is a measure-preserving extension of the Poisson boundary 
of $(G,\mu)$, and thus has maximal $\mu$-entropy.  In particular, an amenable $(G,\mu)$-boundary space is (essentially) isomorphic to the Poisson boundary of $(G,\mu)$. Thus the Poisson boundary is characterized as 
\begin{enumerate}
\item  the unique minimal amenable $(G,\mu)$-space (i.e. it is a factor of every other amenable $(G,\mu)$-space) 
\item the unique maximal $(G,\mu)$-boundary space (i.e. every other $(G,\mu)$-boundary space is a factor of $(B,\nu)$) 
\item  the unique $(G,\mu)$-boundary space which is amenable,   
\item the unique $(G,\mu)$-boundary space of maximal $\mu$-entropy.  

\end{enumerate}

\end{theorem}

\begin{proof}

1) To prove the first characterization,  let us begin by recalling that if $(Y,\nu)$ is any amenable action of $G$ with $\eta$ a quasi-invariant probability measure, and $B$ any compact metric $G$-space, then there exists a $G$-equivariant map $\phi : Y\to P(B)$, where $P(B)$ is the space of probability measures on $B$ and $\phi$ is defined $\eta$-almost everywhere \cite[4.3.9]{zbook}. If $\eta=\nu_Y$ is $\mu$-stationary, its image under $\phi$ denoted $\phi_\ast(\nu_Y)$ is a $\mu$-stationary measure on $P(B)$. 
Now if $(B,\nu)$ is a $\mu$-proximal, then the measure $\phi_\ast(\nu_Y)$ must take values in $\delta$-measures on $B$ almost surely \cite[Ch. VI, Cor 2.10]{M}. Therefore $\phi$ arises from a measurable $G$-equivariant factor map $\phi^\prime(Y,\nu_Y)\to B$. Clearly if the stationary measure $\nu$ on $B$ is unique, then $\phi^\prime_\ast(\nu_Y)$ and $\nu$ must coincide,  so that $(B,\nu)$ is indeed a factor of $(Y,\nu_Y)$ and thus $(B,\nu)$ is indeed a minimal amenable space. 

2) The second characterization is well-known to be valid for any Poisson boundary. 

3) The third characterization follows from the fact that an amenable $(G,\mu)$-boundary space is both a cover and a factor of the Poisson boundary $(B,\nu)$, and hence admits an equivariant endomorphism. But for a $(G,\mu)$-boundary space such an endomorphism is necessarily the identity  \cite[Ch. VI, Cor. 2.10] {M}. Thus the space in question is isomorphic to $(B,\nu)$. 

4) Any $(G,\mu)$-boundary space is a factor of the Poisson boundary $(B,\nu)$, and its $\mu$-entropy 
is bounded by $h_\mu(B,\nu)$. Every proper factor of $(B,\nu)$ has strictly smaller $\mu$-entropy, since 
otherwise the Poisson boundary would be a measure-preserving extension of this factor space (see \cite{nz1}). This is not possible, because every bounded function on $(B,\nu)$ is determined uniquely by its harmonic transform,  so that $(G,\mu)$-boundaries do not admit relatively-measure-preserving factors, 
as these produce distinct functions with the same harmonic transform.  
\end{proof}

\end{document}